\documentclass[12pt]{amsart}\usepackage[myheadings]{fullpage}
\usepackage{fancyhdr}
\usepackage{amsmath,amssymb,amsbsy,amsthm,amsfonts,epsfig,verbatim,enumerate,comment}
\usepackage{graphicx,xcolor,color,wrapfig}
\usepackage{pxfonts,bm}
\usepackage{hyperref}
\input{xy}
\xyoption{all}
\CompileMatrices

\newcommand{\ed}{{\rm d}}
\newcommand{\w}{{\mathchoice{\,{\scriptstyle\wedge}\,}{{\scriptstyle\wedge}}
      {{\scriptscriptstyle\wedge}}{{\scriptscriptstyle\wedge}}}}

\newcommand{\del}{{\partial}}

\newcommand{\delx}{\partial_{\xi}}
\newcommand{\delxb}{\partial_{\xib}}

\newcommand{\liealgebra}[1]{{\mathfrak {#1}}}
\newcommand{\g}{\liealgebra{g}}

\newcommand{\sla}{\liealgebra{sl}}

\newcommand{\su}{\liealgebra{su}}
\newcommand{\un}{\liealgebra{u}}

\newcommand{\liegroup}[1]{{\operatorname{#1}}}
\newcommand{\G}{\liegroup{G}}

\newcommand{\SL}{\liegroup{SL}}
\newcommand{\SO}{\liegroup{SO}}
\newcommand{\Sp}{\liegroup{Sp}}
\newcommand{\SU}{\liegroup{SU}}
\newcommand{\Un}{\liegroup{U}}

\newcommand{\R}{\mathbb R}
\newcommand{\C}{\mathbb C}

\newcommand{\Z}{\mathbb Z}

\newcommand{\PP}{\mathbb P}

\DeclareMathOperator{\Ad}{Ad}

\newcommand{\tr}{\rm tr}

\usepackage{tikz}

\newcommand{\mcc}{\mathcal C}

\newcommand{\mce}{\mathcal E}
\newcommand{\mcf}{\mathcal F}

\newcommand{\mch}{\mathcal H}
\newcommand{\mci}{\mathcal I}

\newcommand{\mcl}{\mathcal L}

\newcommand{\mcn}{\mathcal N}
\newcommand{\mco}{\mathcal O}
\newcommand{\mcp}{\mathcal P}
\newcommand{\mcq}{\mathcal Q}
\newcommand{\mcr}{\mathcal R}

\newcommand{\mcu}{\mathcal U}

\newcommand{\mfj}{\mathfrak J}

\newcommand{\x}{\textnormal{x}}
\newcommand{\xh}{\hat{\textnormal{x}}}
\newcommand{\im}{\textnormal{i}}

\newcommand{\btheta}{\boldsymbol{\theta}}

\newcommand{\bba}{\tb{a}}
\newcommand{\bbb}{\tb{b}}
\newcommand{\bbc}{\tb{c}}

\newcommand{\bbf}{\tb{f}}
\newcommand{\bbg}{\tb{g}}

\newcommand{\bbp}{\tb{p}}

\newcommand{\bbs}{\tb{s}}
\newcommand{\bbt}{\tb{t}}

\newcommand{\ol}{\overline}

\newcommand{\zb}{\ol{z}}

\newcommand{\hb}{\bar h}

\newcommand{\etab}{\ol{\eta}}

\newcommand{\thetab}{\ol{\theta}}
\newcommand{\zetab}{\ol{\zeta}}
\newcommand{\omb}{\ol{\omega}}

\newcommand{\xib}{\ol{\xi}}

\newcommand{\s}[1]{{\mathbb S}^{#1}}
\newcommand{\F}[1]{\mcf^{(#1)}}
\newcommand{\Fh}[1]{\hat{\mcf}^{(#1)}}

\newcommand{\X}[1]{X^{(#1)}}
\newcommand{\Xh}[1]{\hat{X}^{(#1)}}

\newcommand{\I}[1]{{\rm I}^{(#1)}}
\newcommand{\Ih}[1]{{\hat{\rm I}}^{(#1)}}

\newcommand{\cp}[1]{{\C \mathbb{P}^{#1}}}
\newcommand{\pr}{{\mathbb{P}}}
\newcommand{\xinf}{X^{(\infty)}}
\newcommand{\xinfh}{\hat{X}^{(\infty)}}
\newcommand{\iinf}{{\rm I}^{(\infty)}}
\newcommand{\iinfh}{\hat{\rm I}^{(\infty)}}

\DeclareMathOperator{\Gr}{\rm Gr}

\newcommand{\be}{\begin{equation}}
\newcommand{\ee}{\end{equation}}
\newcommand{\benu}{\begin{enumerate}}
\newcommand{\enu}{\end{enumerate}}
\newcommand{\beit}{\begin{itemize}}
\newcommand{\enit}{\end{itemize}}

\newcommand{\bp}{\begin{pmatrix}}
\newcommand{\ep}{\end{pmatrix}}

\newcommand{\lra}{\longrightarrow}

\newcommand{\ff}{{\rm I \negthinspace I}}
\newcommand{\Sigmah}{\hat{\Sigma}}

\newcommand{\n}{\notag}
\newcommand{\noi}{\noindent}

\newcommand{\tn}{\textnormal}
\newcommand{\tb}{\textbf}

\newcommand{\hook}{\hookrightarrow}

\newcommand{\one}{\vspace{1mm}}
\newcommand{\two}{\vspace{2mm}}

\newcommand{\ftmark}{\footnotemark}
\newcommand{\fttext}{\footnotetext}
\newcommand{\sub}{\subsection}
\newcommand{\subb}{\subsubsection}

\makeatletter

\newcommand{\Rmnum}[1]{\expandafter\@slowromancap\romannumeral #1@}
\makeatother

\usepackage[mathscr]{eucal}
\usepackage{bbm}
\usepackage{oldgerm}

\theoremstyle{plain}
\newtheorem{thm}{Theorem}[section]
\newtheorem{lem}[thm]{Lemma}

\newtheorem{cor}[thm]{Corollary}

\newtheorem{prop}[thm]{Proposition}

\newtheorem{rem}[thm]{Remark}

\theoremstyle{definition}
\newtheorem{defn}{Definition}[section]

\newtheorem{exam}[thm]{Example}

\begin{document}
\title[Formal Killing fields for minimal Lagrangian surfaces]
{Formal Killing fields for minimal Lagrangian surfaces\\ in complex space forms}
\author{Joe S. Wang}
\address{Seoul, South Korea}
\email{jswang12@gmail.com}
\date{\today}
\begin{abstract}
The  differential system for  minimal Lagrangian surfaces
in a $2_{\C}$-dimensional, non-flat, complex space form is 
an elliptic system defined on the bundle of oriented Lagrangian planes.
This is a 6-symmetric space associated with the  Lie group $\SL(3,\C)$,
and the minimal Lagrangian surfaces arise as the primitive maps.
Utilizing this property,
we derive the differential algebraic  inductive formulas 
for a pair of  loop algebra $\sla(3,\C)[[\lambda]]$-valued canonical formal Killing fields.
As a result, we give a complete classification of the (infinite sequence of) Jacobi fields
for the minimal Lagrangian system.
We also obtain an infinite  sequence  of higher-order conservation laws
from the components of the formal Killing fields.
\end{abstract}
\maketitle

\setcounter{tocdepth}{2}
\tableofcontents 

\section{Introduction}\label{sec:intro}
\sub{Minimal Lagrangian surface} 
For an immersed Lagrangian submanifold in a K\"ahler manifold,
the vanishing of the mean curvature is equivalent to that 
the associated section of $(n,0)$-form is parallel along the submanifold.
From the well known identity in K\"ahler geometry, 
that the Ricci 2-form is up to constant scale the curvature form of the canonical line bundle,
the minimality condition implies that 
the restriction of Ricci 2-form to the Lagrangian submanifold mush also vanish.
When coupled with the Lagrangian condition, 
they form an over-determined system of differential equations 
which is generally not compatible. 

In case the ambient manifold is K\"ahler-Einstein and 
the Ricci 2-form is a constant multiple of the symplectic form, 
the minimal Lagrangian equation is involutive  and, at least locally,
it admits many solutions, \cite{Bryant1987Lag}.
\subb{Special Lagrangian submanifold}
In relation to the developments in string theory, 
special Lagrangian submanifolds in Calabi-Yau (or Ricci-flat K\"ahler) manifolds
have received much attention recently.
As a particular case of calibrated geometry,
the various aspects of the geometry of special Lagrangian varieties have been studied
from the analytic perspectives, including 
deformation problem \cite{McLean1998}, 
gluing constructions \cite{Haskins2007,Haskins2012,Pacini2013}, 
and singularity analysis  \cite{Joyce2003,Butscher2004}, etc.
We refer to \cite{Joyce2007,Lee2011} for the further  references.

From a different perspective,
the special Lagrangian submanifolds in $\C^m$ with  nontrivial second order symmetries
have been studied by  in-depth analyses of the structure equations in \cite{Bryant2006,Ionel2003}.

\subb{Minimal Lagrangian surface}
In the 2-dimensional case,  Schoen and Wolfson gave a variational analysis of 
area minimizing (Hamiltonian stationary) Lagrangian surfaces, \cite{Schoen2001}.
They proved the existence of area minimizer with the particular forms of admissible conical singularities.
Haskins and Kapouleas gave a gluing construction of 
the compact high-genus special Legendrian surfaces in the 5-sphere, \cite{Haskins2007}. 
Under the Hopf map $\s{5}\to\C\PP^2$, these surfaces are mapped to
minimal Lagrangian surfaces.
In \cite{Loftin2013},
the minimal Lagrangian surfaces in the hyperbolic complex space form $\C\mathbb{H}^2$ 
were studied in relation to
the surface group representations in $\SU(1,2)$.
For the integrable system aspects of the theory on minimal Lagrangian tori, 
we refer to \cite{Carberry2004} and the references therein.
\sub{Formal Killing fields}
\subb{Polynomial Killing field}
One of the characteristic structural properties of a harmonic torus in a symmetric space
is the existence of an associated polynomial Killing field, \cite{Burstall1993}. 
For the case of a minimal Lagrangian torus in $\C\PP^2$,
a polynomial Killing field can be considered as a higher-order Gau\ss\, map
which takes values in the polynomial loop algebra $\sla(3,\C)[\lambda]$
(here $\lambda$ denotes the spectral parameter).
The corresponding spectral curve is a branched triple covering of $\C\PP^1$,
and a minimal Lagrangian torus linearizes on its Jacobian.
From this  construction,
the relevant spectral curve theory of finite type integration   
can be applied to the study of a minimal Lagrangian torus.
\subb{Formal Killing fields}
For a general minimal Lagrangian surface in a $2_{\C}$-dimensional, non-flat, 
complex space form, it turns out that the local analytic data can be packaged into
a pair of canonical \emph{formal} Killing fields,
 which take values in the formal loop algebra 
$\sla(3,\C)[[\lambda]].$ 
For a minimal Lagrangian torus for example,
each of these formal Killing fields would factor  and reduce  to a polynomial Killing field 
up to scaling by an element in $\C[[\lambda^6]]$.

The original idea of canonical formal Killing fields (for CMC surfaces)
is due to Pinkall and Sterling, \cite{Pinkall1989}.
\sub{Results}
\subb{A pair of canonical formal Killing fields}\label{sec:results1}
We give a systematic derivation of the differential algebraic inductive formulas
for a pair of  canonical formal Killing fields 
for the differential system for minimal Lagrangian surfaces  
in a $2_{\C}$-dimensional, non-flat,
complex space form, Thm.\ref{thm:FKformulaep4}, Thm.\ref{thm:FKformulaea5}.
This clarifies the somewhat ad-hoc  recursion formulas appeared  in  \cite{Pinkall1989}\cite{Wang2013}.
\subb{Jacobi fields and pseudo-Jacobi fields}
In the course of analysis,
we find two different kinds of Jacobi fields for the minimal Lagrangian system.

Jacobi fields (ordinary), which correspond to the generalized symmetries 
of the minimal Lagrangian system,
are defined by the operator, \eqref{eq:Jacobieq},
$$\delx\delxb+\frac{3}{2}\gamma^2.$$
Here the notations $\delx, \delxb$ denote  the covariant derivatives
with respect to the unitary $(1,0)$-form $\xi$, and its complex conjugate $\xib$
respectively, \eqref{eq:delxnotation}, 
and $4\gamma^2$ is the holomorphic sectional curvature of the ambient complex space form.
This shows that, when restricted to a minimal Lagrangian surface, 
Jacobi fields are the eigenfunctions of  Laplacian of the induced Riemannian metric
with eigenvalue $6\gamma^2$.
A relevant observation is that the Jacobi operator depends only on the induced metric of the surface,
similarly as in some of the calibrated geometries, \cite{McLean1998}.

Pseudo-Jacobi fields, which correspond to the generalized symmetries
of the elliptic Tzitzeica equation underlying the minimal Lagrangian system,
are defined by the operator, \eqref{eq:Jacobieq'},
$$\delx\delxb+\frac{1}{2}(\gamma^2+4\vert\ff\vert^2).$$
Here $\vert\ff\vert^2$ is the squared norm of the associated Hopf differential, \S\ref{sec:Hopf}.

Applying the previous results from \cite{Fox2011}\cite{Fox2012} 
on the elliptic Tzitzeica equation,
we give a complete classification of the infinite sequence of  (pseudo) Jacobi fields, 
Thm.\ref{thm:classifyJacobi}, 
Cor.\ref{cor:FKformulaep4}, Cor.\ref{cor:FKformulaea5}.
As a corollary, this implies  that
a minimal Lagrangian torus in $\C\PP^2$ admits a pair of  spectral curves.

\subb{Conservation laws}
We also obtain an infinite sequence of higher-order conservation laws
from the components of the canonical formal Killing fields,
Thm.\ref{thm:FKcvlaw}.
 
\sub{Recursion}
We shall give a description of the two 3-step recursion relations embedded in the structure equation 
for the formal Killing fields.
This is the main technical ingredient for our construction of the canonical formal Killing fields.
\subb{Previous works}
The partial differential equation which locally describes the minimal Lagrangian surfaces
in a $2_{\C}$-dimensional, non-flat,  complex space form is the elliptic Tzitzeica equation, \eqref{eq:Tzitzeica}.
An infinite sequence of higher-order symmetries and conservation laws
were determined in the original works \cite{Fox2011}\cite{Fox2012}
via a recursion modeled on the associated formal Killing field equation.
On the other hand, 
this recursion process left the problem of integrating 
a sequence of  exact differential 1-forms.
We show that they can be solved differential algebraically without involving integration.
\subb{Recursive structure equation}
For the case at hand,
a formal Killing field is a (twisted) loop algebra $\sla(3,\C)[[\lambda]]$-valued function 
$\tb{X}_{\lambda}$ on the infinite prolongation space of the minimal Lagrangian system,
which satisfies the Killing field equation \eqref{eq:KillingEquation}.
In terms of an adapted basis of $\sla(3,\C)[[\lambda]]$,
the recursive component-wise structure equation \eqref{eq:formalKilling_n} 
can be summarized in the following infinite schematic diagram of period 6, 
Fig.\ref{fig:diagram0}.
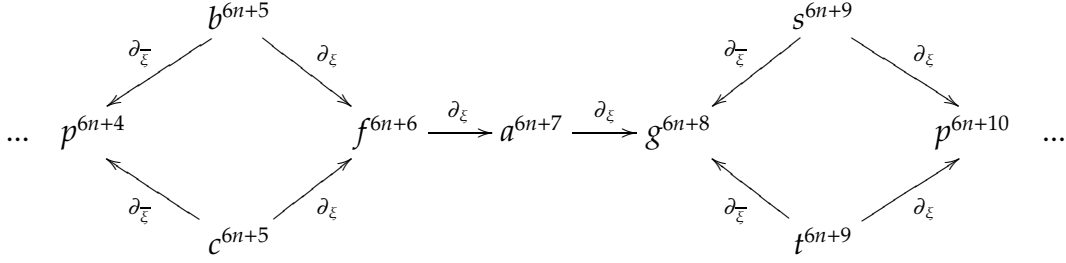
\begin{figure}[t]
\be\label{recursiondiagram}
\begin{split}
\xymatrix{  & b^{6n+5} \ar[ld]_{\delxb} \ar[rd]^{\delx}  & &&
                 & s^{6n+9} \ar[ld]_{\delxb}  \ar[rd]^{\delx}  & \\
  ...\quad p^{6n+4}  & & f^{6n+6}\ar[r]^{\delx} & a^{6n+7}\ar[r]^{\delx}& g^{6n+8} &&  p^{6n+10}\quad... \\
                    & c^{6n+5}\ar[lu]^{\delxb} \ar[ru]_{\delx}  & &&
                    & t^{6n+9}\ar[lu]^{\delxb}\ar[ru]_{\delx}    &   }
\end{split}\n
\ee
\caption{Recursion diagram for  formal Killing field}
\label{fig:diagram0}
\end{figure}

Here the nodes $\{ \,p^*,b^*,c^*,f^*,a^*,g^*,s^*,t^*\,\}$ are the coefficients of a formal Killing field.
Two nodes are connected by an arrow 
only if  there exists a first order differential relation from the structure equation between them.
From a node, one moves to the right by applying $\delx$, 
and to the left by applying $\delxb$.
The upper indices are designated to match (roughly) the jet orders of the coefficients.
 
The structure equation shows that the right-arrows are differential,
which means that the coefficients $\{f^*, a^*, g^*, p^*\}$ are obtained from 
the left-adjacent term(s) by $\delx$ operation.
But, the left-arrows decrease the jet order. In addition, note that
\begin{align}
\delx p^{6n+4}&= \im\gamma b^{6n+5}+2\im h_3 c^{6n+5},\n\\
\delx g^{6n+8}&= - \im \gamma t^{6n+9}-\im h_3 s^{6n+9}.\n
\end{align}
In order to continue the recursion process,
one needs to solve for the coefficients $\{b^*, c^*, s^*, t^*\}$.
\subb{Differential algebraic inductive formulas}
The main idea of construction is to impose the constraint 
in terms of the characteristic polynomial of $\tb{X}_{\lambda}$;
$$\det(\mu\rm{I}_3+\tb{X}_{\lambda})=\mu^3+\tn{c}\lambda^3,$$
for a constant $\tn{c}\in\C^*$.
This allows one to solve for $\{b^*, c^*, s^*, t^*\}$ differential algebraically
not just by using the left-adjacent terms, but by using all of 
the lower-order terms (the left hand side terms in the diagram above).
The relevant explicit formulas using the truncated formal Killing fields
are given in \S\ref{sec:KFformulae}.
\subb{Jacobi fields and conservation laws}
The structure equation \eqref{eq:formalKilling_n} implies that
the sequence of coefficients $\{ \, a^{6n+7} \}$ are Jacobi fields, 
and  the sequence of coefficients $\{ \, p^{6n+4} \,\}$ are  pseudo-Jacobi fields.
From this, it follows that
the 6-step recursion in the schematic diagram in Fig.\ref{fig:diagram0} 
can be understood as
the union of two 3-step recursions between Jacobi fields and pseudo-Jacobi fields.

The structure equation \eqref{eq:formalKilling_n} also implies that
an infinite sequence of conservation laws can be assembled
from the components of the formal Killing fields,
Eq.\eqref{eq:varphin},  Eq.\eqref{eq:varphin'}.

\two
In this way, the canonical formal Killing fields
provide an efficient method to organize the (infinitely prolonged)
local analytic invariants of minimal Lagrangian surfaces.

\sub{Contents}
In \S\ref{sec:setup}, 
the exterior differential system for minimal Lagrangian surfaces is defined
as a homogeneous differential system on the bundle of Lagrangian 2-planes,
and we record the basic structure equations.
In  \S\ref{sec:classicallaws},
we compute the space of classical Jacobi fields and conservation laws 
which arise from the infinitesimal action of the group of 
K\"ahler  isometries of the ambient space form.
In hindsight, the recursion relations 
for the formal Killing fields 
are already implicit in the structure equation for the classical Killing fields.
In \S\ref{sec:prolongation1},
we determine the structure equation for the infinite prolongation of the minimal Lagrangian system.
A branched triple cover is introduced  for the field extension to accommodate 
the higher-order Jacobi fields.
In \S\ref{sec:prolonglemmas},
we record two useful lemmas on the rigidity property of the associated $\delxb$-equation.
They are applied in \S\ref{sec:Jacobifields}
to prove a  complete classification of the (pseudo) Jacobi fields.
\S\ref{sec:formalKilling} contains the main results of the paper.
We use the determinantal identities from the characteristic polynomial 
to derive the explicit differential algebraic recursion for the  canonical formal  Killing fields
associated with  a pair of natural initial data.
In \S\ref{sec:highercvlaws1},
we also read off the formal Killing fields 
an infinite sequence of higher-order conservation laws.
 
\sub{Remarks}
\subb{}
This  is a continuation of the joint work \cite{Wang2013}.
We suspect that the analysis carried out in   \cite{Wang2013}
can be extended to the  primitive harmonic maps in general.
\subb{}
In the analytic approach,
the transition from the minimal surfaces in the 3-sphere
to the minimal Lagrangian surfaces in $\C\PP^2$ is nontrivial, \cite{Schoen2001}.
From our point of view,
this amounts to replacing the underlying Lie algebra from $\sla(2,\C)$ to $\sla(3,\C)$.
The present work may provide a basis to introduce
the further results from the integrable system theory
to the study of minimal Lagrangian surfaces in this uniform perspective.

\section{Minimal Lagrangian surfaces in a complex space form}\label{sec:setup}
After a brief summary of the structure equation for 
a $2_{\C}$-dimensional complex space form,
we give an analytic description of the differential equation for minimal Lagrangian surfaces
as a homogeneous exterior differential system defined on the $\Un(2)/\SO(2)$-bundle  
of oriented Lagrangian planes.
The basic structure  equations established 
in \S\ref{sec:differential system}, together with their infinite prolongation 
in \S\ref{sec:prolongation1},
will be the basis of our analysis for the minimal Lagrangian system.
\subsection{Complex space form}\label{sec:cspaceform}
We will summarize the basic formulas of K\"ahlerian geometry for a $2_{\C}$-dimensional complex space form. We refer to \cite{Bryant2001} for the further related details. 

In order to avoid repetitions, we agree on the following range for the indices:
\[ 1\leq A, B, C \leq  2.\]
We use the Einstein summation convention for repeated indices.
\subb{$2_{\C}$-dimensional complex space form}\label{sec:2dcspaceform}
Let $M$ be the $2_{\C}$-dimensional simply connected complex space form of constant holomorphic sectional curvature $4\gamma^2$. In the case $4\gamma^2=0$ and $M=\C^2$, it turns out that the minimal Lagrangian surfaces in $M$ are equivalent to the holomorphic curves in $\C^2$ under a different covariant constant complex structure (this is explained by the fact that $\C^2$ is hyperK\"ahler, \cite[p.148]{Joyce2000}).  
We shall 
restrict ourselves to the case
$$\fbox{$\quad \gamma^2\ne0,\quad$}$$
where the differential equation for minimal Lagrangian surfaces is genuinely nonlinear.

The squared expression $\gamma^2$ is introduced 
for the sake of convenience, for the quantity $\gamma$ (which is $\frac{1}{2}$-times the square root of the holomorphic sectional curvature) appears frequently in the analysis.
We adopt the following convention for $\gamma$:
\begin{equation}\label{1deRham1}
\gamma=
\begin{cases}
&+\sqrt{\gamma^2}\\
&+\im \sqrt{- \gamma^2}
\end{cases}\quad \tn{if}\quad
\begin{array}{l}
 \mbox{$\gamma^2>0$} \\
 \mbox{$\gamma^2<0$.}
\end{array} \n
\end{equation}
Here $\im=\sqrt{-1}$ denotes the unit imaginary number.

\subb{Unitary coframe bundle}\label{sec:unitarybundle}
Let $\Un(2)$ be the group of 2-by-2 unitary matrices. Let\two

\centerline{\xymatrix{\Un(2) \ar[r] & \mathcal{F} \ar[d]^{\pi} \\ & M  }}

\two\noi be the principal $\Un(2)$-bundle of unitary coframes.  An element $\mathfrak{u}\in\mcf$ is by definition a Hermitian isometry
\[\mathfrak{u}:T_{\pi(\mathfrak{u})}M\to \C^2,\]
where $\C^2$ is the standard $2_{\C}$-dimensional Hermitian vector space.
The structure group $\Un(2)$ acts on $\mcf$ on the right by
\[\mathfrak{u} \to g^{-1}\circ\mathfrak{u},\quad \mbox{for}\;g\in\Un(2).
\]

Let $(\zeta^1,\,\zeta^2)^t$ be the $\C^2$-valued tautological 1-form on $\mcf$.
The structure group $\Un(2)$ acts on $(\zeta^1,\,\zeta^2)^t$ on the right by
\[  \bp \zeta^1 \\ \zeta^2 \ep  \to g^{-1}  \bp \zeta^1 \\ \zeta^2  \ep,
\quad \mbox{for}\;g\in\Un(2).
\]
By definition, the K\"ahler structure on $M$ is given by the pair
\be\label{1varpi} \begin{array}{rll} 
 \tn{g}:=&   \zeta^A\circ\zetab^A &\tn{(Riemannian metric)},\\
\varpi:=& \frac{\im}{2}  \zeta^A\w \zetab^A  &\tn{(symplectic form)}.
\end{array}\ee
 
\one
Let $\un(2)$ be the space of 2-by-2 skew-Hermitian matrices, which is the Lie algebra of $\Un(2)$. 
There exists a unique $\un(2)$-valued connection 1-form $(\zeta^A_B)$ on $\mcf$ 
such that  the following structure equations hold:
\begin{align}\label{1strt1}
\ed \zeta^A &= - \zeta^A_B \w \zeta^B, \\
\zeta^A_B&=-\zetab^B_A. \n
\end{align}
The curvature 2-forms $\Omega^A_B$ are then defined by 
\be\label{1strt11}
\Omega^A_B=\ed \zeta^A_B+\zeta^A_C\w \zeta^C_B.
\ee
For the case at hand, the curvature forms of the complex space form $M$ are given by
\be\label{eq:1curv}
\Omega^A_B= \gamma^2\left( \zeta^A\w \zetab^B
+\delta^A_{B}\sum_{C=1}^2 \zeta^C\w \zetab^C    \right).
\ee
Here $\delta^A_{B}$ is the Kronecker delta.

\subb{Real structure equation}\label{sec:realstrt}
The above treatment considers $M$ as a $2_{\C}$-dimensional complex Hermitian manifold. 
On the other hand, note the isomorphism
$$\Un(2)=\SO(4)\cap\Sp(2,\R). $$
Here the Lie groups $\Un(2), \SO(4)$ $\tn{(special orthogonal group)}$, 
$\Sp(2,\R)$ (symplectic group) are considered as the real subgroups of $\SL(4,\R)$.
Accordingly, it will be convenient for the analysis of Lagrangian surfaces to consider $M$ as a $4_{\R}$-dimensional real manifold equipped with the pair $(\tn{g}, \varpi)$ given by \eqref{1varpi}.
 
\two
To this end, we decompose the structure equations for $\{ \zeta^A, \zeta^A_B\}$  into the real, and imaginary parts as follows.  

Set
\begin{align}
\zeta^A&=\omega^A+\im \mu^A, \n\\
\bp \zeta^1_1 & \zeta^1_2 \\  \zeta^2_1 & \zeta^2_2 \ep
&=\bp \cdot & \rho \\ -\rho &\cdot \ep
+\im \bp \beta^1-3\gamma^2\theta_0 & \beta^2 \\ \beta^2 & -\beta^1-3\gamma^2\theta_0 \ep,  \n
\end{align}
for the set of real 1-forms $\{ \omega^A, \mu^A, \rho, \beta^A, \theta_0\}$.
In terms of these 1-forms,
the structure equations \eqref{1strt1}, \eqref{1strt11},  \eqref{eq:1curv} 
are written as follows:
\begin{align}\label{1strt0}
\ed\omega^1&=-\rho\w\omega^2+(\beta^1\w\mu^1 +\beta^2\w\mu^2)-3 \gamma^{2}\theta_0\w\mu^1, \\
\ed\omega^2&=+\rho\w\omega^1+(\beta^2\w\mu^1-\beta^1\w\mu^2)-3 \gamma^{2}\theta_0\w\mu^2, \n\\
\ed\mu^1&=-\rho\w\mu^2-(\beta^1\w\omega^1 +\beta^2\w\omega^2)+3 \gamma^{2}\theta_0\w\omega^1, \n\\
\ed\mu^2&=+\rho\w\mu^1-(\beta^2\w\omega^1-\beta^1\w\omega^2)+3 \gamma^{2}\theta_0\w\omega^2,\n \\ 
\n\\
\ed\rho&= \gamma^2(\omega^1\w\omega^2+\mu^1\w\mu^2)+2\beta^1\w\beta^2,  \n\\
\n\\
\ed\theta_0&= -(\mu^1\w\omega^1+\mu^2\w\omega^2), \n\\
\ed\beta^1&=-2\rho\w\beta^2+\gamma^2(\mu^1\w\omega^1-\mu^2\w\omega^2),\n\\
\ed\beta^2&=+2\rho\w\beta^1+\gamma^2(\mu^2\w\omega^1+\mu^1\w\omega^2).\n
\end{align}

\subsection{Exterior differential system}\label{sec:differential system}
With this preparation, we proceed to describe the differential system for minimal Lagrangian surfaces.
\subb{Bundle of Lagrangian 2-planes}\label{sec:Lagbundle}
An immersed  surface $\x: \Sigma\hook M$ is Lagrangian if 
$$\x^*\varpi=0.$$
This is by definition a first order constraint on $\Sigma$ 
that its tangent space at each point is a Lagrangian subspace of the tangent bundle $TM$. 

Under the standard representation, the unitary group $\Un(2)$ acts transitively on the set of oriented Lagrangian subspaces ($2$-planes) in $\C^2$, 
with $\SO(2)=\Un(2)\cap \SL(2,\R)$ as the stabilizer subgroup. 
The Grassmannian of Lagrangian subspaces  in dimension 2 is 
the homogeneous space,
$$\tn{Lag}(\C^2)=\Un(2)/\SO(2).$$
From  the general theory of principal bundles, it follows that 
the $\Un(2)/\SO(2)$-bundle of oriented Lagrangian 2-planes in $TM$ is given by
$$X:=\mcf/\SO(2)\to M.$$

This is summarized by the following diagram:

\two\two
\centerline{\xymatrix{ &&\mcf \ar[dl]_-{\SO(2)}\ar[dd]^{\Un(2)}  \\
  & X \ar[dr]_-{\Un(2)/\SO(2)}  &   \\ & &M}}\two\two

\subb{Differential system for  Lagrangian surfaces}\label{sec:Lagsystem}
In view of the diagram, 
an immersed oriented Lagrangian surface in $M$ admits 
a unique tangential lift to $X$. 
Such tangential lifts are characterized as the integral surfaces 
of the canonical contact differential system on $X$.

This is expressed analytically as follows. Consider the real structure equation \eqref{1strt0}.
By a standard moving frame analysis,
set the contact ideal $\mci_0$ on $X$ generated by
\begin{equation}\label{1ideal0}
\mci_0:=\langle\, \mu^1, \mu^2,   \ed\mu^1, \ed\mu^2 \,\rangle.
\end{equation}
Note that the differential ideal $\mci_0$  is originally defined on $\mcf$.
But, it is invariant under the induced action by the structure group $\SO{(2)}$ 
and $\mci_0$ descends on $X$ as a well defined differential ideal.  
By construction,
an immersed oriented  integral surface of $\mci_0$ in $X$ corresponds to 
a possibly singular oriented Lagrangian surface in $M$.

An immersed integral surface of $\mci_0$ in $X$ may become singular under the projection $X\to M$ to the original complex space form. 
The formulation of Lagrangian surfaces in $M$ 
in terms of the canonical contact differential system $\mci_0$ on $X$ 
naturally extends the set of admissible Lagrangian surfaces, see \S\ref{sec:doublecover}.

\subb{Differential system for  minimal Lagrangian surfaces}\label{sec:mLagsystem}
The condition that an immersed Lagrangian surface  is minimal (i.e., its mean curvature vector vanishes) is a second order constraint. We therefore need to augment the contact ideal $\mci_0$ to express this additional minimality condition.

From Eqs.\eqref{1strt0}, we claim that this is equivalent to the vanishing of the connection 1-form $\theta_0$,  $$\theta_0=0.$$ 
To see this, consider the structure equation \eqref{1strt0} adapted to a Lagrangian surface 
$\Sigma\hook M$.  With the 1-forms $\{\mu^1, \mu^2\}$ being set to 0,  
the  induced Riemannian metric on  $\Sigma$ is given by
\be
\rm{I}:=(\omega^1)^2+(\omega^2)^2. \n
\ee
For the condition of minimality,  note that the second fundamental form of $\Sigma$ can be identified with the symmetric cubic differential
\be 
\underline{\ff}:=  \omega^A \circ \tn{Im}(\zeta^A_B) \circ \omega^B. \n
\ee
The mean curvature vector vanishes when the corresponding trace vanishes, 
$$\tr_{\rm{I}}(\underline{\ff})=0.$$
This is equivalent to the vanishing of the 1-form $\theta_0$, which is up to constant scale the trace of  $\tn{Im}(\zeta^A_B).$
\begin{prop}\label{prop:mLagdefi}
Let $X\to M$ be the $\Un(2)/\SO(2)$-bundle of oriented Lagrangian 2-planes.
Let $\mci_0$ be the canonical contact differential system on $X$, \eqref{1ideal0}.
The differential system for minimal Lagrangian surfaces is given by
\begin{equation}\label{1ideal1}
\mci:=\langle\, \mu^1, \mu^2, \theta_0,   \phi^+, \phi^- \,\rangle,
\end{equation}
where
\begin{align}
\phi^+&=\beta^1\w\omega^1+\beta^2\w\omega^2, \n \\
\phi^-&=\beta^2\w\omega^1-\beta^1\w\omega^2. \n
\end{align}
The structure equation \eqref{1strt0} shows that  the differential ideal $\mci$, originally defined on $\mcf$, is invariant under the induced action by the structure group $\SO{(2)}$. 
The differential ideal $\mci$ is  well defined on $X$. 

An immersed minimal Lagrangian surface in $M$ admits a unique tangential lift to $X$ 
as an integral surface of $\mci$.
Conversely, an immersed integral surface of $\mci$ in $X$ projects to 
a possibly singular (branched) minimal Lagrangian surface in $M$. 
\end{prop}
\begin{proof}
For the last sentence, see \S\ref{sec:doublecover}. 
\end{proof}
Note from the structure equation that 
$$\ed \theta_0, \ed\mu^A \equiv 0\mod\mci,$$
and $\mci$ is differentially closed.

In terms of Cartan's theory of exterior differential systems,   
the differential system $(X,\mci)$ is involutive and the local moduli space of solutions depends on two arbitrary real functions of 1 variable, \cite{Bryant1991}.
 
\subsubsection{Complexified structure equation}\label{sec:complexify}
The differential system for minimal Lagrangian surfaces  under consideration 
is an integrable extension over the elliptic Tzitzeica equation, \cite{Fox2011}.
The  characteristic directions are complex and they induce a complex structure 
on a minimal Lagrangian surface.  
In order to utilize this property, we introduce another set of 
complex differential forms on $\mcf$ adapted to $\mci$.

\two
Set
\begin{align}\label{eq:cforms}
\xi&:=\omega^1+\im\,\omega^2,   \\
\theta_1&:=\mu^1-\im\mu^2, \n\\
\eta_2&:=\beta^1-\im\beta^2.\n
\end{align}
Note
\[\phi^+  - \im \phi^-=\eta_2\w\xi,
\]
and the ideal \eqref{1ideal1} can be written as
\be\label{1ideal2}
\mci=\langle \theta_0,  \theta_1,  \eta_2\w\xi \rangle.
\ee

In terms of these complex 1-forms,  the structure equation \eqref{1strt0} simplifies to
the following set of equations:
\begin{align}\label{eq:strt2}
\ed \xi &= \im \rho \w \xi -3\gamma^2 \theta_0\w\thetab_1-\theta_1\w\etab_2,\\
\ed \theta_0&=-\frac{1}{2}\left( \theta_1\w\xi+\thetab_1\w\xib  \right),   \n\\
\ed \theta_1 &=-\im\rho\w\theta_1-\eta_2\w\xi+3\gamma^2\theta_0\w\xib,\n \\
\ed \eta_2 &=-2\im\rho\w\eta_2+\gamma^2\theta_1\w\xib, \n \\
\ed \rho&=\frac{\im}{2}\left(\gamma^2\xi\w\xib-2\eta_2\w\etab_2-\gamma^2\theta_1\w\thetab_1\right). \n
\end{align}
From now on, the differential analysis for minimal Lagrangian surfaces will be carried out based on this structure equation.
 
\two
Let us mention here a relevant notation which will be frequently used. 
For a scalar function $f:\mcf\to\C$, the covariant derivatives are 
written in the upper-index notations,  
\be\label{eq:notation0}
\ed f \equiv f^{\xi}\xi+f^{\xib}\xib+f^0\theta_0+f^1\theta_1+f^{\bar{1}}\thetab_1+f^2\eta_2+f^{\bar{2}}\etab_2
\mod \rho.
\ee

\subsection{Hopf differential}\label{sec:HopfBonnet}
The induced local geometric structures on a minimal Lagrangian surface 
consists of a triple of data called admissible triple, Defn.\ref{defn:Bonnettriple} in the below.
In particular, it contains a holomorphic cubic differential called \emph{Hopf differential}
which arises as a complexified version of the second fundamental form.
Hopf differential will play an important role 
in our study of the  minimal Lagrangian system.
\subb{Hopf differential}\label{sec:Hopf}
Let $\x:\Sigma\hook X$ be an immersed integral surface of the differential system for minimal Lagrangian surfaces.
This is expressed analytically by  
$$\theta_0, \theta_1=0, \; \eta_2\w\xi=0,$$
on the induced $\SO(2)$-bundle $\x^*\mcf\to \Sigma$.
It follows that the structure equation \eqref{eq:strt2} restricted to $\x^*\mcf$ becomes
\begin{align}\label{eq:strtx}
\ed\xi&=\im \rho\w\xi,\\
0&=\eta_2\w\xi, \n \\
\ed\eta_2&=-2\im \rho\w\eta_2, \n \\
\ed\rho&=\frac{\im}{2}\left(\gamma^2\xi\w\xib - 2\eta_2\w\etab_2\right).\n
\end{align} 
In particular, the cubic differential
$$\eta_2\circ \xi^2$$
is invariant under the action by the structure group $\SO(2)$ and becomes a well defined holomorphic cubic differential on $\Sigma.$ 

\two
Let $K\to\Sigma$ denote the canonical line bundle of $(1,0)$-forms.
\begin{defn}\label{Hopf}
Let $\x: \Sigma \hook X$ be an immersed integral surface of the differential system for minimal Lagrangian surfaces.  Consider the induced structure equation \eqref{eq:strtx}.
The \tb{Hopf differential} of $\x$  is the holomorphic cubic differential
\begin{equation}\label{eq:Hopf}
\ff=\eta_2 \circ \xi^2  \in H^0(\Sigma, K^3).
\end{equation}
The \tb{umbilic divisor} $\mcu =(\ff)_0$ is defined as the zero divisor of $\ff$. 
\end{defn}
When $\Sigma$ is compact, by Riemann-Roch theorem we have  
$$\tn{deg}(\mcu)=6\, \tn{genus}(\Sigma)-6.$$
\subb{Admissible triple}\label{sec:Bonnet}
Suppose that the integral surface $\Sigma\hook X$ is the tangential lift of an immersed oriented minimal Lagrangian surface in $M$. 
By definition of the tautological forms on $\mcf$, this implies the independence condition 
$$\xi\w\xib\ne 0,$$ 
and the induced Riemannian metric on $\Sigma$ is given by
$$\rm{I}=\xi\circ\xib.$$
It follows  that  $\x^*\mcf\to \Sigma$  can be identified with 
the principal $\SO(2)$-bundle of the induced metric, and 
that $\xi$ is the tautological unitary $(1,0)$-form,
and $\rho$ is the Levi-Civita connection form. 

From the second equation of \eqref{eq:strtx},  there exists a coefficient $h_3$ defined on $\x^*\mcf$ such that
\be\label{eq:etah3}\eta_2= h_3 \xi.\ee
The Hopf differential is now written as $$\ff=h_3\xi^3.$$
 
From the fourth equation of \eqref{eq:strtx}, the Gau\ss\, equation,
the curvature $R$ of the induced metric $\rm{I}$ satisfies the compatibility equation
\be\label{eq:Gauss}
R=\gamma^2- 2 h_3\hb_3.  
\ee

\one
Summarizing the analysis so far, we give a definition of the compatible Bonnet data for minimal Lagrangian surfaces.
\begin{defn}\label{defn:Bonnettriple}
Let $M$ be a $2_{\C}$-dimensional complex space form of constant holomorphic sectional curvature $4\gamma^2.$
An \tb{admissible triple} for a  minimal Lagrangian surface in $M$ consists of  a Riemann surface, a conformal metric, and a holomorphic cubic differential which satisfy the compatibility equation \eqref{eq:Gauss}.
\end{defn}

An analogue of the classical Bonnet theorem  can be stated as follows. 
The proof is by a standard ODE argument and is omitted.
\begin{thm}\label{thm:Bonnet}
Let $\Sigma$ be a Riemann surface. Let $( \rm{I}, \ff)$ be a pair of a conformal metric and a holomorphic cubic differential on $\Sigma$ such that $(\Sigma, \rm{I}, \ff)$ form an admissible triple for a minimal Lagrangian surface in a complex space form $M$. Let $\pi: \widetilde{\Sigma}\to\Sigma$ be the simply connected universal cover.  Then there exists a minimal Lagrangian  immersion $\tilde{\x}: \widetilde{\Sigma}\hook M$ which realizes $(\pi^*\rm{I}, \pi^*\ff)$ as the induced Riemannian metric and the Hopf differential.
Such an immersion $\tilde{\x}$ is unique up to motion 
by the K\"ahler isometries of the ambient complex space form $M$.
\end{thm}
\subsubsection{Cube root of Hopf differential}
The analysis of the canonical formal Killing fields to be addressed in \S\ref{sec:formalKilling} 
will inevitably involve the use of the object
$$\sqrt[3]{h_3},$$
or, equivalently, the cube root of Hopf differential
$$\sqrt[3]{\ff}.$$
It is generally a multi-valued holomorphic 1-form on a minimal Lagrangian surface.
In order to better accommodate this,
we introduce the triple cover of a minimal Lagrangian surface defined by Hopf differential.
\begin{defn}\label{defn:doublecover}
Let $\x:\Sigma\hook X$ be an immersed integral surface of the differential system for minimal Lagrangian surfaces. Let $\ff\in H^0(\Sigma, K^3)$ be the Hopf differential \eqref{eq:Hopf}.
The \tb{triple cover} 
$$\nu:\Sigmah\to\Sigma$$ 
associated with $\ff$ is the Riemann surface of the complex curve
\[
\Sigma'= \left\{  \kappa  \in K\; \vert \; \kappa^3=\ff \right\} \subset K.
\]

The \tb{cube root $\omega=\sqrt[3]{\ff}$} of the Hopf differential is the holomorphic 1-form on $\Sigmah$ obtained by the pull-back of  restriction of the tautological 1-form on $K$ to $\Sigma'$.
By construction, the projection $\nu$ is a triple covering branched over
the umbilics of degree $1$ or $2\mod 3$.  
\end{defn} 

 
\section{Classical Killing fields} \label{sec:classicallaws} 
Let $$G=\SU(3),  \quad \tn{or}\quad \SU(1,2)$$ be the group of 
K\"ahler isometries of the complex space form $M$, 
depending on the sign of the holomorphic sectional curvature   
$4\gamma^2.$ 
The induced action of $G$ on $X$ is transitive,
and  by construction 
the differential system $(X,\mci)$ is $G$-invariant, i.e.,
the group $G$ acts as a symmetry of the differential system  for minimal Lagrangian surfaces.
In this section, we examine the  classical Killing fields 
generated by the corresponding infinitesimal  action of the Lie algebra of $G$,
$$\g = \su(3), \quad \tn{or}\quad\su(1,2).$$ 

\one
The structure equation for classical Killing fields reveals its recursive, symmetrical  property
when it is written in terms of an adapted basis of $\g$, 
Eq.\eqref{eq:classicalKF}.
From this,  we extract two kinds of Jacobi fields called \emph{classical Jacobi fields}, and
their variants \emph{classical pseudo-Jacobi fields},
as well as the recursion relations between them, \S\ref{sec:recursionrel}.

The structure equation  also shows that 
there exists the eight dimensional space of classical conservation laws. 
Combining these results, 
we deduce that Noether's theorem holds 
and there exists a canonical isomorphism 
between the  classical (pseudo) Jacobi fields and 
the  classical (local) conservation laws, Cor.\ref{cor:Noether}. 

\sub{Structure equation}\label{sec:classicalstrt}
The Killing field equation, \eqref{eq:KillingEquation0} in the below, characterizes 
the classical Killing fields
generated by the infinitesimal action of the Lie algebra $\g$ on 
the homogeneous space $X$.
When restricted to a minimal Lagrangian surface, it becomes the recursive structure equation \eqref{eq:classicalKF}. 

We will find that Eq.\eqref{eq:classicalKF} encodes many of the important local structural properties of the minimal Lagrangian system. In particular, 
a pair of the embedded recursion relations   will serve as the model for 
the construction of the canonical formal Killing fields.
\subb{Complexified Maurer-Cartan form}
Recall the complexified differential forms \eqref{eq:cforms} and their structure equation  \eqref{eq:strt2}.
The left invariant, $\g$-valued Maurer-Cartan form of $G$ can be written in terms of these 1-forms by
\be\label{eq:psiform}
\psi=\psi_{+}+\psi_{0}+\psi_{-},
\ee
where
\be
\psi_{-} =\frac{1}{2} \left[ \begin {array}{ccc} 
\cdot &-\gamma(\xi-\im\theta_1) & \im \gamma( \xi+\im \theta_1)  \n \\
 \gamma(\xi+\im\theta_1) &\im \eta_2 & -  \eta_2  \\
 -\im\gamma(\xi-\im\theta_1) &- \eta_2 &- \im \eta_2 \end {array} \right],
\ee 
\be \psi_{0} =
\left[ \begin {array}{ccc} 
2\im\gamma^2\theta_0 &\cdot & \cdot  \n \\
\cdot &-\im\gamma^2\theta_0& \rho   \\
\cdot  & -\rho &-\im\gamma^2\theta_0
\end {array} \right], \n
\ee
\be \psi_{+}=\frac{1}{2} \left[ \begin {array}{ccc} 
\cdot &-\gamma(\xib-\im\thetab_1) &- \im \gamma( \xib+\im \thetab_1)  \n \\
 \gamma(\xib+\im\thetab_1) &\im \etab_2 &   \etab_2  \\
 \im\gamma(\xib-\im\thetab_1) & \etab_2 &- \im \etab_2 \end {array} \right].
\ee $\,$

\noi   It satisfies the structure equation $$\ed\psi+\psi\w\psi=0.$$
\subb{Classical Killing fields}
\begin{defn}
Let $\mcf/\SO(2)=X\to M$ be the bundle of oriented Lagrangian planes.
A \tb{classical Killing field} is an $\SO(2)$-equivariant,
$\g$-valued function $\tb{X}$ on $\mcf$ which satisfies the Killing field equation 
\be\label{eq:KillingEquation0}
\ed \tb{X} +[\psi, \tb{X}]=0.
\ee
\end{defn}
By definition, the space of classical Killing fields is  isomorphic to $\g$. 

\begin{defn}
Let $(X,\,\mci)$ be the differential system for minimal Lagrangian surfaces. A vector field $V \in  H^0(TX)$ is a \tb{classical symmetry} if it preserves the differential ideal $\mci$ under the Lie derivative,
$$\mcl_V \mci \subset \mci.$$
The algebra of classical symmetries is denoted by $\mathfrak{S}^{(0)}.$
\end{defn}
It is clear that a classical Killing field generates a classical symmetry.
In fact,  a classical symmetry is necessarily generated by a classical Killing field. We thus have;
\begin{prop}\label{prop:classicalsymmetrytocv}
$$ \mathfrak{S}^{(0)}\simeq \g.$$
\end{prop}
The proof follows by a direct computation, and is omitted.
\subb{Recursive structure equation}
For the ensuing analysis, it will be useful to have Eq.\eqref{eq:KillingEquation0} written component-wise.
Consider the explicit decomposition of the Lie algebra $\g$ in Fig.\ref{fig:Xdecompo0}.\\ 

\begin{figure}[h]
$$
\left[ \begin {array}{ccc} -2\im \tb{a} &\tb{b}+\tb{f}+\tb{g}-\tb{s}
&\im \tb{b}- \im \tb{f}+\im \tb{g}+\im \tb{s}\\
\noalign{\medskip}-\tb{b}+\tb{f}+\tb{g}+\tb{s} & \im \tb{c}+\im \tb{a}-\im \tb{t}&-\tb{p}+\tb{c}+\tb{t}\\
\noalign{\medskip}-\im \tb{b}-\im \tb{f} +\im \tb{g}-\im \tb{s}
&\tb{p}+\tb{c}+\tb{t}&-\im \tb{c}+\im \tb{a}+\im \tb{t}\end {array} \right]
$$
\caption{Decomposition of the Lie algebra $\g=\su(3),\;\tn{or}\;\, \su(1,2)$}
\label{fig:Xdecompo0}
\end{figure}

\noi
Here $\{ \, \tb{p}, \tb{b}, \tb{c}, \tb{f}, \tb{a}, \tb{g}, \tb{s}, \tb{t}   \,\}$ are the scalar coefficients which satisfy the following reality conditions:
\be
\g=\begin{cases}
\su(3)    \\ \su(1,2)  
\end{cases} \tn{then}\quad \begin{aligned} 
&\tb{a}=\ol{\tb{a}},\tb{p}=\ol{\tb{p}},  \tb{t}=-\ol{\tb{c}}, \tb{s}=-\ol{\tb{b}},  \tb{g}=-\ol{\tb{f}},  \\
&\tb{a}=\ol{\tb{a}},\tb{p}=\ol{\tb{p}},  \tb{t}=-\ol{\tb{c}}, \tb{s}=+\ol{\tb{b}},  \tb{g}=+\ol{\tb{f}}.
\end{aligned}\n
\ee

The Killing field equation \eqref{eq:KillingEquation0} then implies
\begin{align}\label{eq:classicalKF}
 \ed  \tb{p}  &\equiv(\im\gamma  \tb{b}+2\im h_3 \tb{c}  )\xi
                          + ( \im\gamma \tb{s}+2\im \hb_3 \tb{t})\xib,     \\
 \ed   \tb{b} +\im \tb{b}\rho
                         &\equiv\im h_3  \tb{f} \xi+ \frac{\im}{2}\gamma  \tb{p} \xib,    \n\\
 \ed   \tb{c} -2 \im \tb{c}\rho&\equiv \im \gamma  \tb{f}\xi  +\im \hb_3  \tb{p}  \xib,    \n\\
 \ed  \tb{f} -\im  \tb{f}\rho  &\equiv\frac{3\im}{2}\gamma \tb{a} \xi
                                                          +(\im \gamma  \tb{c}+\im\hb_3  \tb{b})\xib, \n\\
 \ed   \tb{a}  &\equiv\im \gamma  \tb{g}\xi + \im \gamma  \tb{f}\xib,  \n\\
 \ed  \tb{g}+\im \tb{g}\rho &\equiv (-\im \gamma \tb{t}-\im h_3  \tb{s})\xi
                                                             +\frac{3\im}{2}\gamma \tb{a}\xib,   \n\\
 \ed  \tb{s} -\im  \tb{s}\rho&\equiv \frac{\im}{2} \gamma  \tb{p}\xi -\im\hb_3 \tb{g}\xib,  \n\\
 \ed   \tb{t}+2 \im \tb{t}\rho&\equiv\im h_3  \tb{p}\xi-\im \gamma \tb{g}\xib, \n\qquad\mod\mci.
\end{align}
Here we applied the substitution, \eqref{eq:etah3},
$$\eta_2\equiv h_3\xi,\quad \etab_2\equiv \hb_3\xib \qquad("\tn{mod}\;\mci").$$
The resulting structure equation should be understood as 
 restricted to a minimal Lagrangian surface.
 
\two
Let us introduce the relevant notations for partial derivatives. 
For a scalar function $u$ on $\mcf$, the notations $\delx u, \delxb u$ would mean
\begin{align}\label{eq:delxnotation}
\delx u =u_{\xi}    &:=\tn{ $\xi$-coefficient of $\ed u$}, \\
\delxb u =u_{\xib}&:=\tn{ $\xib$-coefficient of $\ed u$, \; "\tn{mod}\;$\mci$}".\n
\end{align}
For instance, the above structure equation shows that
$$ \tb{a}_{\xi}=\im\gamma\tb{g}, \quad \tb{a}_{\xib}=\im\gamma\tb{f}.$$

\subsection{Classical Jacobi fields and pseudo-Jacobi fields}\label{sec:classicalJacobi}
From the apparent recursive symmetry of Eq.\eqref{eq:classicalKF}, 
we extract two different kinds of Jacobi fields  for the minimal Lagrangian system;
classical Jacobi fields and classical pseudo-Jacobi fields.
\subb{Classical Jacobi field}
Consider the coefficient $\tb{a}$ in Eq.\eqref{eq:classicalKF}. One finds
\begin{align*}
\tb{a}_{\xi}&=\im\gamma\tb{g}, \\
\tb{a}_{\xi,\xib}&=\im\gamma\tb{g}_{\xib}=-\frac{3}{2}\gamma^2\tb{a}. 
\end{align*}
Motivated by this, we make the following definition.
\begin{defn}\label{defn:classicalJacobifield}
A scalar function $A:X\to\C$ is a \tb{classical Jacobi field} if it satisfies
\be\label{eq:classicalJacobifield}
A_{\xi,\xib}+\frac{3}{2}\gamma^2A=0.  
\ee
The $\C$-vector space of classical Jacobi fields is denoted by $\mathfrak{J}^{(0)}=\mathfrak{J}^0$.
\end{defn}
An examination of Eq.\eqref{eq:KillingEquation0} shows that 
the coefficient $\bba$ generates the classical  Killing field $\tb{X}$ in the sense that 
if the $\tb{a}$-component vanishes then the corresponding classical Killing field $\tb{X}$ vanishes. It follows that the induced map
$$\bba:\g^{\C} \lra \mathfrak{J}^{(0)}$$
is injective. Here $\g^{\C}=\g\otimes\C=\sla(3,\C)$ is the complexification of $\g$.

We claim  that this is in fact an isomorphism, and
the $\tb{a}$-components of the (complexified) classical Killing fields 
generate the classical Jacobi fields.  
\begin{prop}\label{prop:classicalJg}
\be\label{eq:J0=g}
\mathfrak{J}^{(0)}\simeq \g^{\C}.
\ee
\end{prop} 
The claim follows by a direct computation.
Rather than a complete proof, 
we give a brief  sketch of ideas.

Let $A$ be a classical Jacobi field, which is defined on $X$. 
Denote the covariant derivatives of $A$ by
$$\ed A=A^{\xi}\xi+A^{\xib}\xib+A^0\theta_0+A^1\theta^1+A^{\ol{1}}\thetab_1
+A^2\eta_2+A^{\ol{2}}\etab_2.$$
We adopt the similar notation for the successive derivatives of $A^{\xi}, A^{\xi, \xib},   \, ...$, etc.

Applying the covariant derivative operators \eqref{eq:delxnotation}, one finds
\begin{align*}
\delx A=A_{\xi}&=A^{\xi}+A^2 h_3,\\
\delxb(A_{\xi})=A_{\xi, \xib}&=A^{\xi,\xib}+A^{\xi,\ol{2}}\hb_3+(A^{2,\xib}+A^{2,\ol{2}}\hb_3) h_3.
\end{align*}
The Jacobi equation \eqref{eq:classicalJacobifield}  becomes
\be\label{eq:separateJB}
\Big(A^{\xi,\xib}+\frac{3}{2}\gamma^2A\Big)
+A^{\xi,\ol{2}}\hb_3+ A^{2,\xib}h_3+A^{2,\ol{2}}\hb_3 h_3=0.\ee

Since $A$ is defined on $X$ while $h_3,\hb_3$ are the prolongation variables,\ftmark\fttext{See \S\ref{sec:prolongation1}.} 
each coefficient of $\{ h_3,\hb_3, h_3\hb_3\}$ in \eqref{eq:separateJB} must vanish separately. As a result,  a classical Jacobi field must satisfy;
$$A^{\xi,\xib}+\frac{3}{2}\gamma^2A=0, \; A^{\xi,\ol{2}}=A^{2,\xib}=A^{2,\ol{2}}=0.$$
From this, a standard over-determined PDE analysis shows that $A$ is the $\bba$-component 
of a (complexified) classical Killing field.  

\subb{Classical pseudo-Jacobi fields}
In an analogy with the classical Jacobi field case, consider next the coefficient $\tb{p}$ in Eq.\eqref{eq:classicalKF}. One finds
\begin{align*}
\tb{p}_{\xi}&=\im\gamma  \tb{b}+2\im h_3 \tb{c} , \\
\tb{p}_{\xi,\xib}&=\im\gamma  \tb{b}_{\xib}+2\im h_3 \tb{c}_{\xib} 
=-\frac{1}{2}(\gamma^2+4 h_3\hb_3)\tb{p}. 
\end{align*}
\begin{defn}\label{defn:classicalpseudoJacobifield}
A scalar function $P:X\to\C$ is a \tb{classical pseudo-Jacobi field} if it satisfies
\be\label{eq:classicalpsJacobifield}
P_{\xi,\xib}+\frac{1}{2}(\gamma^2+4 h_3\hb_3)P=0.  
\ee
The $\C$-vector space of classical pseudo-Jacobi fields is denoted by 
$\mathfrak{J}'^{(0)} =\mathfrak{J}'^{0}$.
\end{defn}
By the similar argument  as above,  the $\tb{p}$-components of 
the (complexified) classical Killing fields generate the classical pseudo-Jacobi fields. 
\begin{prop}\label{prop:classicalpsJ=g}
\be\label{eq:J0ps=g}
\mathfrak{J}'^{(0)}\simeq \g^{\C}.
\ee
\end{prop}
As a consequence of the isomorphisms \eqref{eq:J0=g}, \eqref{eq:J0ps=g}, we have
\begin{cor}\label{cor:J0J0ps=g}
$$\mathfrak{J}^{(0)}\simeq\mathfrak{J}'^{(0)}\simeq \g^{\C}.$$
\end{cor}
\subb{Recursion relations}\label{sec:recursionrel}
The structure equation \eqref{eq:classicalKF}
contains two 3-step recursion relations between
the classical pseudo-Jacobi field $\bbp$ and the classical Jacobi field $\bba$.
We record and emphasize these structures here,
for they are the main technical ingredients 
in the construction of the formal Killing fields later on.

\one
[\tb{From $\bbp$ to $\bba$}]

Let the classical pseudo-Jacobi field $\bbp$ be given. 
From \eqref{eq:classicalKF}, suppose (either of) the equations
$$\delxb\bbb=\frac{\im}{2}\gamma\bbp,\quad \delxb\bbc=\im \hb_3\bbp$$
were solved.  
Differentiating $\bbb, \bbc$,  
$$\delx\bbb=\im h_3\bbf,\quad\delx\bbc=\im\gamma\bbf,$$
and one gets $\bbf$.
Differentiating $\bbf$,  
$$\delx\bbf=\frac{3\im}{2}\gamma\bba,$$
and one gets the classical Jacobi field $\bba$.

The process can be summarized by the following diagram.
$$\bbp  \xrightarrow{\;\;\;\;\delxb^{-1}\;\;\;\;}  \bbb, \bbc 
 \xrightarrow{\quad\delx\quad} \bbf \xrightarrow{\quad\delx\quad} \bba.
$$
 
\two
[\tb{From $\bba$ to $\bbp$}]

In a similar way, starting from the  classical Jacobi field $\bba$,
one may reach $\bbp$ by the following process.
$$\bba \xrightarrow{\quad\delx\quad} \bbg \xrightarrow{\;\;\;\;\delxb^{-1}\;\;\;\;} 
\bbs, \bbt \xrightarrow{\quad\delx\quad} \bbp.
$$

\two
Note that by combining the two processes, 
one gets a 6-step recursion relation for classical (pseudo) Jacobi fields.

\subsection{Classical conservation laws}\label{sec:classicalcvlaw}
Another important invariants of the minimal Lagrangian system
are conservation laws.
They also can be read off Eq.\eqref{eq:classicalKF}.
\subb{Definition}
Let $(\Omega^*(X), \ed)$ be the de-Rham complex of $\C$-valued differential forms on $X$.
Since $\mci$ is a differential ideal closed under the exterior derivative,
the quotient complex
\[(\underline{\Omega}^*,\, \underline{\ed}) \]
is well defined, where  $\underline{\Omega}^*=\Omega^*(X)/\mci$,  
and $\underline{\ed}=\ed\mod \mci$.
Let $H^{q}(\underline{\Omega}^*,\, \underline{\ed})$ be the cohomology   
at $\underline{\Omega}^q.$  The set 
$$\{ \;H^{q}(\underline{\Omega}^*,\, \underline{\ed}) \;\}_{q=0}^2$$
 is called the \emph{characteristic cohomology} of the differential system $(X,\mci)$.
\begin{defn}\label{defn:classicalcvlaw}
Let $(X,\mci)$ be the differential system for minimal Lagrangian surfaces.
A \tb{classical conservation law}  is  an element of the 1-st characteristic cohomology $H^1(\underline{\Omega}^*,\, \underline{\ed})$ of $(X,\mci)$. The $\C$-vector space of classical conservation laws is denoted by
\[ \mcc^{(0)}= \mcc^{0}:=H^1(\underline{\Omega}^*,\, \underline{\ed}).\]
Let $\mcc^{(0)}_{loc}$ denote  the space of local classical conservation laws of $\mci$
restricted to a small contractible open subset of $X$.
\end{defn}

\subb{Classical conservation laws from classical Killing fields}
From the structure equation \eqref{eq:classicalKF}, consider the 1-form
\be\label{eq:cvlawphi_a}
\varphi_{\tb{a}}=\tb{b}\xi+\tb{s}\xib.
\ee
One finds that
$$\ed \varphi_{\tb{a}}\equiv 0\mod\mci,$$
and the 1-form $\varphi_{\tb{a}}$ represents a classical conservation law. 

Let $[\varphi_{\tb{a}}]\in\mcc^{(0)}_{loc}$ denote the class represented by $\varphi_{\tb{a}}$
(which is globally defined on $X$).
We claim that the associated map
$$\g^{\C} \simeq \mfj^{(0)} \lra \mcc^{(0)}_{loc}$$ given by
$$\tb{a} \lra [\varphi_{\tb{a}}]$$
is an isomorphism. 

In order to verify this claim, and for other computational purposes, we introduce a differentiated version of conservation laws.

\sub{Classical differentiated conservation laws}\label{sec:cdcvlaw}
From the exact sequence
$$0\to  \mci  \to \Omega^*(X)\to \Omega^*(X)/\mci=\underline{\Omega}^*\to 0,$$
at least locally we have
$$\mcc^{(0)}_{loc}\simeq H^2(\mci, \ed).$$
Here $H^*(\mci,\ed)$ denotes the cohomology of the complex $(\Omega^*(\mci), \ed).$
\begin{defn}\label{defn:classicaldcvlaw}
Let $(X,\mci)$ be the differential system for minimal Lagrangian surfaces.
A \tb{classical differentiated conservation law} is an element in the $2$-nd  cohomology 
$H^2(\mci, \ed).$
The $\C$-vector space of classical differentiated conservation laws is denoted by
$$  \mch^{(0)}=\mch^0:= H^2(\mci, \ed).$$
\end{defn}
Note from the definition of classical conservation laws that a representative 1-form $\varphi$ of a class $[\varphi]\in\mcc^{(0)}_{loc}$ is defined up to exact 1-forms. By taking the exterior derivative ``$\ed$" and considering the differentiated conservation laws, we eliminate this ambiguity, which is practically  important in the actual computation of  conservation laws. 

Moreover, in a geometric situation as in the present case, the ideal $\mci$ is often equipped with the relevant structures which enable one to determine a subspace of $\Omega^2(\mci)$ in which a differentiated conservation law has the unique representative. In other words, one of the advantages of the differentiated version of conservation laws is that 
it may lead to an analogue of Hodge theorem;
the appropriately normalized differentiated conservation laws can be considered as the \emph{harmonic} representatives.
\begin{rem}
Consider the symplectic form of $M$, \eqref{1varpi}:
\begin{align*}
\varpi&= \frac{\im}{2}  \zeta^A\w \zetab^A \\
&=\omega^1\w\mu^1+\omega^2\w\mu^2 = \ed\theta_0.
\end{align*}
Thus the class $[\varpi]$ is trivial in $H^{2}(\mci,\,\ed).$
\end{rem}
We shall  first establish the isomorphism
$$\mch^{(0)}\simeq\g^{\C} $$
by a direct differential analysis. Then, by finding a section (up to constant scale) of the natural map
\be\label{eq:varphiPhimap}
[\varphi_{\tb{a}}]\in\mcc^{(0)}_{loc} \lra  [\ed \varphi_{\tb{a}}]\in\mch^{(0)},\ee
we show that there exists an   isomorphism
$$\mcc^{(0)}_{loc}\simeq \mch^{(0)}.$$
\subb{Initial reduction}
Let $\Phi\in  \Omega^2(\mci)$ be a 2-form in the ideal.  Up to addition by $\ed( \Omega^1(\mci))$, an exact 2-form in the ideal of the form $\ed(f_0\theta_0+f_1\theta_1+f_{\bar{1}}\thetab_1)$ for scalar coefficients 
$f_0, f_1, f_{\bar{1}}$, a computation shows that one may write
\be\label{eq:H0}
\Phi=A \Psi + \theta_0\w\sigma,
\ee
where $A$ is a scalar function, 
$\sigma\in\Omega^1(X)$, and
$$\Psi=\tn{Im}(\theta_1\w\xi)=-\frac{\im}{2}(\theta_1\w\xi-\thetab_1\w\xib).
$$
Note
\begin{align}\label{1dPsi}
\ed\Psi&= 3\im\gamma^2\theta_0\w(\xi\w\xib+\theta_1\w\thetab_1)     \\
&\equiv 0\mod\theta_0.     \n
\end{align}

A 2-form $\Phi\in\Omega^2(\mci)$ normalized as in \eqref{eq:H0} is called 
a \emph{reduced 2-form}.
Let $$H^{(0)}\subset  \Omega^2(\mci)$$ denote the subspace of  such reduced 2-forms.
By construction, $H^{(0)}$ is transversal to the subspace of exact 2-forms
$\ed( \Omega^1(\mci))\subset  \Omega^2(\mci)$. It follows that  we have an isomorphism
\be
\mch^{(0)} \simeq \{ \,\tn{closed 2-forms in }\, H^{(0)} \, \}.\n
\ee

\subb{Structure equation}
The equation for the classical differentiated conservation laws is now reduced to
$$\ed\Phi=0,$$
for $\Phi\in H^{(0)}$.
From this, a direct differential analysis yields the following closed structure equation.\ftmark\fttext{This part of the analysis is more or less mechanical, and shall be omitted.}
\be\label{eq:classicalJacobi}
\ed \begin{pmatrix} A\\ A^{\xi}\\A^{\xib} \\ A^1\\A^{\bar{1}}\\A^{1,1}\\A^{\bar{1}, \bar{1}}\\P \end{pmatrix} =
\im \begin{pmatrix} \cdot \\-A^{\xi}\\A^{\xib}\\A^1\\-A^{\bar{1}} \\ 2A^{1,1}\\-2A^{\bar{1}, \bar{1}}\\   \cdot \end{pmatrix}\rho
+\begin{pmatrix}
A^{\xi}& A^{\xib}&\cdot&  A^1& A^{\bar{1}}&\cdot&\cdot\\
A^{\bar{1}, \bar{1}}&-\frac{3}{2}\gamma^2A & -3\gamma^2 A^{\bar{1}} & P& \cdot&
A^{1}&\cdot \\
-\frac{3}{2}\gamma^2A & A^{1,1}& -3\gamma^2 A^1 &\cdot   & -P &
\cdot  & A^{\bar{1}}\\
P&\cdot&  3\gamma^2 A^{\xib} &A^{{1,1}}&-\frac{3}{2}\gamma^2A&
 \cdot &-A^{\xi} \\
\cdot& -P &  3\gamma^2 A^{\xi} &-\frac{3}{2}\gamma^2A& A^{{\bar{1},\bar{1}}}&
 -A^{\xib} &\cdot  \\
 -\gamma^2 A^{\xib}& \cdot&\cdot&\cdot&-\gamma^2 A^1 &
\cdot &-2P \\
  \cdot&-\gamma^2 A^{\xi}&\cdot&-\gamma^2 A^{\bar{1}} &\cdot&
2P & \cdot   \\
\frac{\gamma^2}{2}A^{\bar{1}} &-\frac{\gamma^2}{2}A^{1}&\cdot
&\frac{\gamma^2}{2}A^{\xib}&-\frac{\gamma^2}{2}A^{\xi} &
A^{1,1} & -A^{\bar{1},\bar{1}}  \end{pmatrix}
\begin{pmatrix} \xi \\ \xib \\ \theta_0\\\theta_1\\ \thetab_1 \\ \eta_2 \\ \etab_2 \end{pmatrix}.
\ee
\noi Here the upper-script $A^{*,*}$ denotes the covariant derivative as before.

One finds that the closed linear differential system for  
$\{A,A^{\xi},A^{\xib},A^1,A^{\bar{1}}, A^{1,1}, A^{\bar{1},\bar{1}}, P\}$ is compatible, i.e.,
$\ed^2=0$ is a formal identity.
By the existence and uniqueness theorem of ODE,  the space of (local) classical differentiated   conservation laws is eight dimensional.

Note that,  at this stage, the reduced 2-form $\Phi$ is normalized to
\be\label{eq:H0AA}
\Phi=A \Psi +\im \theta_0\w\left(A^{\xi}\xi-A^{\xib}\xib+A^{1}\theta_1-A^{\bar{1}}\thetab_1\right).  
\ee

\sub{Noether's theorem}
Observe that the coefficients of $\Phi$ are determined by the coefficient $A$ (or $P$) and its successive derivatives. Moreover, Eq.\eqref{eq:classicalJacobi} shows that
$$\begin{cases}
&\tn{$A$ is a classical Jacobi field},\\
&\tn{$P$ is a classical pseudo-Jacobi field.}
\end{cases}$$
It follows from this  that we have the isomorphisms
$$ \mch^{(0)}\simeq
\mathfrak{J}^{(0)}\simeq \mathfrak{J}'^{(0)}\simeq \g^{\C}.$$
In particular, this implies that 
the classical differentiated conservation laws are globally defined on $X$.
\begin{cor}[\tb{Noether's theorem} for classical conservation laws]\label{cor:Noether} 
Let $M$ be the simply connected $2_{\C}$-dimensional complex space form of constant holomorphic sectional curvature $4\gamma^2.$
Let $G=\SU(3),\,\tn{or}\; \SU(1,2)$ be the group of K\"ahler isometries of  $M$.
Let $\g$ be its Lie algebra.
Then there exist  the  isomorphisms
$$ (\mathfrak{S}^{(0)})^{\C}\simeq \mathfrak{J}^{(0)}\simeq \mathfrak{J}'^{(0)}\simeq\mch^{(0)}\simeq\g^{\C}\simeq\mcc^{(0)}_{loc}.$$
Here $\mcc^{(0)}_{loc}$ denotes the space of local  conservation laws of $\mci$
restricted to a small contractible open subset of $X$.
\end{cor}
\begin{proof}
For a classical Jacobi field $A$, set
$$\varphi_A=\frac{\im}{6\gamma^2}\left( A^{\xi}\thetab_1-A^{\xib}\theta_1+A^{\bar{1}}\xi-A^1\xib  \right).
$$
Then $ \ed\varphi_A=\Phi$, where $\Phi$ is given by \eqref{eq:H0AA}.
This establishes the  isomorphism $\mch^{(0)}\simeq\mcc^{(0)}_{loc}.$
\end{proof}
For the global consideration of conservation laws including the contribution from 
the cohomology
of the ambient space $X$, we refer to  \cite[p.584, Theorem 1]{Bryant1995}.
\begin{rem}[Moment conditions]
Consider a closed integral curve $\Gamma\subset X$ of $\mci$.
Such $\Gamma$  corresponds to a pair $(\gamma, \Pi)$ of a closed curve $\gamma\subset M$ 
equipped with a field of oriented Lagrangian 2-planes $\Pi$ tangent to $\gamma$
such that the associated $(2,0)$-vector field is covariant constant. 
The integrals of the classical conservation laws on $\Gamma$ 
yield the moment conditions for $(\gamma, \Pi)$ 
to bound a minimal Lagrangian surface, 
with the given boundary $\gamma$ and the prescribed  tangent planes $\Pi$ along $\gamma$.
Compare \cite{Fu1995}.
\end{rem}

\subsection{Higher-order extension}
In the remainder of the paper,  the analysis of the classical objects in this section will be extended to 
their higher-order analogues in the setting of
 the  infinite prolongation of the differential system $(X,\mci)$.

\two
 The extension process relies on the familiar extension of the  Maurer-Cartan form 
$$\psi \lra \psi_{\lambda}=\lambda \psi_{+}+\psi_0+\lambda^{-1}\psi_{-},$$
obtained by inserting the spectral parameter $\lambda$. 
This leads to the generalization of the $\g$-valued classical Killing field $\tb{X}$ 
to the loop algebra $\g^{\C}[[\lambda]]$-valued formal Killing field $\tb{X}_{\lambda},$
$$\tb{X}\lra \tb{X}_{\lambda}.$$
In practice, this amounts to \emph{expanding} the scalar coefficients $\{ \tb{p}, \tb{b}, \tb{c}, \tb{f}, \tb{a}, \tb{g}, \tb{s}, \tb{t}  \}$ of $\tb{X}$ to the appropriate formal series in $\lambda$.

We will find that this extension is also adapted so that  
the structure equation for the formal Killing field contains
the higher-order version of the recursion relations discussed in  
\S\ref{sec:recursionrel}.
We will be able to read off the higher-order (pseudo) Jacobi fields and conservation laws 
from the  structure equation for  $\tb{X}_{\lambda}$.

\two
As a first step to the higher-order analysis, we begin by introducing 
the infinite prolongation of the differential system $(X,\mci)$,
and record the basic structure equations.

\section{Prolongation}\label{sec:prolongation1}
The minimal Lagrangian surfaces under consideration 
 are locally described by the elliptic Tzitzeica equation,
which is a well known example of an  integrable elliptic equation.
It possesses an infinite sequence of higher-order symmetries 
and conservation laws, \cite{Fox2012}.
In order to access the corresponding structures for the  minimal Lagrangian system,
it becomes necessary to introduce the infinite prolongation.   
 
\two
In \S\ref{sec:prolonginf}, 
we determine  the basic structure equations for the infinite prolongation.
The commutation relations for the dual frame of vector fields will be used in the next sections
for the classification of Jacobi fields.
In \S\ref{sec:doublecover}, 
we define the   triple cover of the  infinite prolongation
to support the triple cover of a minimal Lagrangian surface defined by Hopf differential.
In \S\ref{sec:balanced}, 
we introduce a sequence of adapted functions called \emph{balanced coordinates}
and record their basic structure equations. 

\two
In hindsight, the hidden symmetries of the minimal Lagrangian system begin to emerge  
in the form of a weighted homogeneity with respect to the balanced coordinates.
To be precise, this is a non-local symmetry associated with 
an integrable extension.
We do not pursue to formulate this properly in this paper, 
and the notion of weighted homogeneity 
would suffice for our purposes.  
We mention that
this is a part of an infinite hierarchy of non-local symmetries called spectral symmetry.

\subsection{Infinite prolongation}\label{sec:prolonginf}

\subb{Infinite sequence of $\PP^1$-bundles}\label{sec:infP1}
Set $(X^{(0)},\,\I{0})=(X,\,\mci)$ be the original differential system \eqref{1ideal2}.
Inductively define  $(\X{k+1},\,\I{k+1})$ as the prolongation of $(\X{k},\,\I{k})$ such that
$$\pi_{k+1, k}: \X{k+1}\to \X{k}$$
is the bundle of oriented integral 2-planes of $(\X{k},\,\I{k})$, and  that
the differential ideal $\I{k+1}$ is generated by $\pi_{k+1,k}^*\I{k}$ 
and the restriction of the canonical contact ideal to $\X{k+1}\subset\Gr^+(2,T\X{k})$.

For each $k\geq 0$, the prolongation space $\X{k+1}$ is a smooth manifold. 
The projection $\pi_{k+1,k}$ is a smooth submersion with two dimensional fibers isomorphic to $\C\mathbb{P}^1$, see \S\ref{sec:doublecover} for further details.
\begin{defn}
The \tb{infinite prolongation} $(\xinf,\I{\infty})$ of the differential system $(X,\mci)$ for minimal Lagrangian surfaces is defined as the projective limit
\begin{align}
 \xinf &=\lim_{\longleftarrow}  \X{k}, \n \\
\I{\infty}&=\cup_{k\geq 0} \I{k}. \n
\end{align}
Let $\pi_{\infty,k}:\xinf\to\X{k}$ be the associated projection.
Here we identity $\I{k}$ with its image $\pi_{\infty,k}^*\I{k}\subset\Omega^*(\xinf).$
By construction, the sequence of Pfaffian systems $\I{k}$ satisfy the inductive closure conditions 
$$\ed\I{k}\equiv 0\mod \I{k+1}, \; k\geq 1.$$
\end{defn}

\subb{Associated principal $\SO(2)$-bundles}\label{eq:associatedbundle}
For the purpose of computation, we also introduce the following associated bundles.

Consider  the principal $\SO(2)$-bundle  $\Pi:\mcf\to X=\mcf/\SO(2)$. 
Set $\F{k}=\Pi^*\X{k}$ for $k\geq 1$, and define
\[\F{\infty}= \lim_{\longleftarrow} \,\F{k}.\]
The higher-order differential analysis  will be practically carried out on 
the principal $\SO(2)$-bundle
$$ \F{\infty}\to\X{\infty}$$
in an $\SO(2)$-equivariant manner so that it has a well defined meaning on $\X{\infty}$.
For simplicity, we continue to use $\I{k}, \iinf$ to denote 
the corresponding differential  ideals on $\F{k}, \F{\infty}.$

\subb{Zariski open sets}\label{sec:opensets}
Let $\X{1}_0\subset\X{1}$ be the open subset defined by the independence condition
\be\label{eq:X10}
\X{1}_0=\{ \;\, \tn{E}\in \X{1} \;\vert\; \; (\xi\w\xib)_{\vert_{\tn{E}}}\ne 0\,\}.
\ee
For $k\geq 1$, inductively define the corresponding sequence of open subsets 
$$\X{k+1}_0=\pi_{k+1,k}^{-1}(\X{k}_0).$$ 
Let $$\lim_{\longleftarrow} \X{k}_0=\xinf_0\subset\xinf$$ 
denote their projective limit.
Let $\F{k}_0\subset\F{k}, \F{\infty}_0\subset\F{\infty}$
denote the associated open subsets.

Dually, consider the complementary open subset
\be\label{eq:X1infty}
\X{1}_{\infty}=\{ \;\, \tn{E}\in \X{1} \;\vert\; \;  (\eta_2\w\etab_2)_{\vert_{\tn{E}}}\ne 0\,\}.
\ee
Its associated open subsets 
$\X{k}_{\infty}, \xinf_{\infty}, \F{k}_{\infty},  \F{\infty}_{\infty}$ are defined similarly as above.
Note that the pair  $\{  \xinf_{0},  \xinf_{\infty}\}$ form an open covering of $\xinf$ 
(and similarly for $\{  \F{\infty}_{0},  \F{\infty}_{\infty}\}$ and $\F{\infty}$).

\two
The analysis of the infinitely prolonged differential system for minimal Lagrangian surfaces 
will be carried out on $\F{\infty}_0$, which is one half of the open covering of $\F{\infty}.$ 
But we remark that it can be equally well carried out on the other half of the open covering $\F{\infty}_{\infty}$.
Since the formulation of the structure equation on $\F{\infty}_{\infty}$ is almost identical to that on $\F{\infty}_0$, we shall present only the $\F{\infty}_{0}$-part of the analysis and omit the $\F{\infty}_{\infty}$-part.

Let us give instead an indication on how the dual prolongation process on $\F{1}_{\infty}$ would proceed.
Given the 2-form  $\eta_2\w\xi$  in the ideal $\mci$, 
recall the prolongation variable $h_3$ defined by \eqref{eq:etah3}.
Introduce the new prolongation variable $p_3$ such that 
$$\xi=p_3\eta_2$$
on the integral 2-planes in $\F{1}_{\infty}$ (and hence $p_3=h_3^{-1}$ on $\F{1}_{0}\cap\F{1}_{\infty}$).
By switching from $h_3$ to $p_3$, it is clear that one may proceed analogously to the infinite prolongation on $\F{\infty}_{\infty}$, such that the associated formulas agree on the intersection 
$\F{\infty}_{0}\cap\F{\infty}_{\infty}.$
  
\subb{Infinitely prolonged structure equation}\label{sec:infstrt}
Recall the induced structure equation \eqref{eq:strtx} on an immersed 
minimal Lagrangian surface.
By the standard prolongation process of the exterior differential system 
theory,  
\cite[Chapter \rm{VI}]{Bryant1991}, 
set
$$\eta_2=h_3\xi$$
for a prolongation variable $h_3.$  From the third equation of \eqref
{eq:strtx}, we get
$$\ed h_3+3\im h_3\rho\equiv 0\mod\xi.$$

Inductively define the higher-order derivatives of   $h_3$   by the 
equations
\be\label{eq:dhj}
\ed h_j +  \im   j    h_j   \rho = h_{j+1}  \xi + T_{j}  \xib,
\; \; j = 3, \, 4, \, ... \, ,
\ee
where
\begin{align}\label{eq:Tj}
T_{3}&=0, \,     \\
T_{j+1}&=   \sum_{s=0}^{j-3}
a_{j,s} \,  h_{j-s} \, \partial^{s}_{\xi} R,    \; \; \; \mbox{for} 
\; \; j\geq 3,  \n   \\
& \quad  \;\;a_{j,s} =\frac{(j+2s+3) }{2(j-1)} {j-1 \choose s+2}, 
\n\\
& \quad  \partial^{s}_{\xi}  R=\delta_{0s}\gamma^2 - 2 h_{3+s} 
\hb_3.\n
\end{align}
The formula for the sequence  $\{ \,T_{j+1}\,\}$ is  uniquely 
determined by requiring that
$\ed(\ed h_j)=0$ for $j=3,\,4,\,...\,$. This implies the recursive 
relation
\be \label{eq:Trecurs}
T_{j+1} = \partial_{\xi} T_j + \frac{j}{2}R\, h_j.
\ee
For example, 
$$T_3=0,  \;  T_4=\frac{3}{2}h_3(\gamma^2-2h_3\hb_3),\;
T_5=\frac{7}{2}\gamma^2h_4 -10h_3\hb_3h_4.$$
\begin{rem}
Similarly as in \eqref{eq:delxnotation},
we use the subscript notation $f_{\xi}$ or $\,\del_{\xi}f$ to denote 
the $\xi$-coefficient of $\ed f$ (and similarly for $f_{\xib}$ or  $
\del_{\xib}f$).
\end{rem}

Based on Eqs.\eqref{eq:dhj}, \eqref{eq:Tj}, we  define the following set of differential forms and  vector fields on $\F{\infty}_0$:

\two
\begin{enumerate}[\qquad a)]
\item
Differential forms:
\be\label{eq:thetaj}
\begin{array}{rlllll}
\eta_j&=\ed h_j&+\im j h_j \rho&&-T_j \xib,\,\quad&\tn{for}\;j\geq 3,\\
\theta_j&=\eta_j &&-h_{j+1}\xi, &&\tn{for}\; j\geq 2. \\
\end{array}
\ee
On the open subset $\,\F{k}_0$, the  set of 1-forms
$$\{ \rho; \, \xi, \xib, \theta_0, \theta_1 , \thetab_{1}, \ldots, \,\theta_{k+1}, \thetab_{k+1};
\,\eta_{k+2}, \etab_{k+2}\}$$
form a coframe. On the open subset $\,\F{\infty}_0$, the  set of 1-forms
\be\label{Finfframe}
\{ \rho; \, \xi, \xib, \theta_0, \theta_1 , \thetab_{1},  \,\theta_2, \thetab_{2}, \ldots \}
\ee
form a coframe.
By construction,
$$\iinf=\langle \theta_0,  \theta_1 , \thetab_{1},  \,\theta_2, \thetab_{2}, \ldots \rangle.$$

\one \item
Vector fields:
\begin{align}
\delx&:=   \tn{total derivative with respect to}\;\xi\mod\iinf,\n\\
\delxb&:= \tn{total derivative with respect to}\;\xib\mod\iinf.\n
\end{align}
On the open subset $\,\F{\infty}_0$,  let
\be\label{Finfframev}
\{ E_{\rho}; \, \delx=E_{\xi}, \delxb=E_{\xib}, E_0, E_1 , E_{\bar{1}}, E_2,E_{\bar{2}}, \ldots \}
\ee
denote the dual frame of \eqref{Finfframe}.
\end{enumerate}

\two
In terms of the 1-forms introduced above, we denote 
the covariant derivatives of  $T_{j}$ by
\be
\ed T_{j} +\im \, (j-1) T_{j} \rho=(\del_{\xi} T_{j}) \xi + (\del_{\xib} T_{j}) \xib +T_{j}^{\bar{3}}\thetab_{3}
+\sum_{s=3}^{j-1}T_{j}^s\theta_s,\quad \,j\geq 3,\n
\ee
where $T_j^s=E_s(T_j)$.

 
Recall from \eqref{eq:strt2},
\begin{align}\label{eq:strt22}
\ed \xi &= \im \rho \w \xi -3\gamma^2 \theta_0\w\thetab_1-\theta_1\w\thetab_2-\hb_3\theta_1\w\xib,  \\
\ed \theta_0&=-\frac{1}{2}\left( \theta_1\w\xi+\thetab_1\w\xib  \right),   \n\\
\ed \theta_1 +\im\rho\w\theta_1&=-\theta_2\w\xi+3\gamma^2\theta_0\w\xib,\n \\
\ed \theta_2+2\im\rho\w\theta_2 &=-\theta_3\w\xi +3h_3\gamma^2\theta_0\w\thetab_1
+h_3\theta_1\w\thetab_2+(\gamma^2+h_3\hb_3)\theta_1\w\xib, \n \\
\ed \rho&=\frac{\im}{2}\left(R\xi\w\xib-2\theta_2\w\thetab_2-\gamma^2\theta_1\w\thetab_1
-2\hb_3\theta_2\w\xib+2h_3\thetab_2\w\xi \right). \n
\end{align}
Extending this, 
we record the structure equations satisfied by the 1-forms $\theta_j$  on $\,\mcf^{(\infty)}_0$.
\begin{lem}\label{lem:prolongedstrt}
For  $\,j\geq 3$,
\begin{align}\label{eq:strt2inf}
\ed \theta_j + \im j \rho\w\theta_j&=-\theta_{j+1}\w\xi
+3\gamma^2  \theta_0\w(T_j\theta_1+h_{j+1}\thetab_1)
+\frac{jh_j}{2}\left( \gamma^2 \theta_1\w\thetab_1+2\theta_2\w\thetab_2 \right)  \\
&\quad+T_j\thetab_1\w\theta_2+h_{j+1}\theta_1\w\thetab_2
+\tau_j'\w\xi+\tau_j''\w\xib,  \quad\tn{where} \n \\
&\quad\tau_j'=  h_3T_j \thetab_1- j h_3h_j  \thetab_2,\n\\
&\quad\tau_j''= \hb_3 h_{j+1}\theta_1+j \hb_3h_j  \theta_2 -(\,T_{j}^{\bar{3}}\thetab_{3}+\sum_{s=3}^{j-1}\,T_{j}^s\theta_s).\n
\end{align}
\end{lem}
\begin{proof}
Direct calculation. We omit the details.
\end{proof}

Eqs.\eqref{eq:strt22}, \eqref{eq:strt2inf} imply the following commutation relation for the dual frame.
\begin{cor}\label{lem:commutation}
The dual frame \eqref{Finfframev} satisfies the following commutation relations.
\be\label{eq:lxib}
[E_{\ell},\,E_{\xib} ] =\sum_{j\geq {\ell}+1}T^{\ell}_j E_j, \qquad \ell\geq 3.
\ee
\end{cor}
\begin{proof}
Let $\theta$ be a differential 1-form, and let $E_a,\,E_b$ be vector fields.
The corollary follows from Cartan's formula
\[ \ed\theta(E_a,\,E_b)=E_a(\theta(E_b))-E_b(\theta(E_a))-\theta([E_a,\,E_b]).
\]
\end{proof}
We introduce the following notations  in order to keep track of orders  ($k\geq 2$):
\begin{itemize}
\item $\mco(k)$: functions on $\xinf$ that do not depend on $h_j$ 
for $j>k$.
\item $\mco(-k)$: functions on $\xinf$ that do not depend on $\hb_j$ 
for $j>k$.
\end{itemize}
\noi
Note for example that 
$$ T_j\in\mco(j-1). $$
\noi

\subsection{Triple cover}\label{sec:doublecover}
\subb{Set up}\label{sec:triplemoti}
The triple cover $\Sigmah\to\Sigma$ of an immersed integral surface of $(X,\mci)$ defined in 
Defn.~ \ref{defn:doublecover} prompts the definition of a global triple cover 
$\hat{X}^{(\infty)}\to\xinf$ such that the following commutative diagram holds:
\one\\
\centerline{\xymatrix{\Sigmah \ar[d]  \;  \ar@{^{}.>}[r] &\;\;\hat{X}^{(\infty)} \ar[d] \\
                                \Sigma  \;  \ar[r]  &\;\;\xinf  }} \\

\two\noi
It turns out that the triple cover $\xinfh$ to be constructed also supports the splitting field
for the  (pseudo) Jacobi equation  for the minimal Lagrangian system, 
see \S\ref{sec:Jacobifields}.

\subb{Definition}\label{sec:tripledefi}
Let
$$K^{\mcf} \to \mcf, \quad K^{\mcf}_{\eta_2} \to \mcf,$$
be the (trivial) complex line bundles generated by the 1-forms $\xi, \eta_2$ respectively.
Recall the principle $\SO(2)$-bundle $\Pi:\mcf \to X$, and the structure equation \eqref{eq:strt2}.
Let
$$K\to X, \quad K_{\eta_2} \to X,$$
be the  induced line bundles respectively.
Note from Eq.\eqref{eq:strt2} that
$$ K_{\eta_2} =K^{-2}.$$

Consider $\mcf^{(1)}=\Pi^*(\X{1})$. Since an element in $\mcf^{(1)}$ is defined by the equation
$$\eta_2-h_3\xi=0$$
for a coefficient $h_3$, we have
$$\mcf^{(1)}\cong \pr( K^{\mcf} \oplus  K^{\mcf}_{\eta_2}) \cong \mcf \times \cp{1}.$$
It follows that
$$ \X{1}\cong \pr(K^{p+1} \oplus K^{p-2})$$
for any integer $p$.
In order to define the triple cover, we choose $\X{1} \cong \pr(K^{3} \oplus \C)$.
\begin{defn}
The \tb{triple cover} $\Xh{1}$ of the first prolongation $\X{1}$ is defined by
\[ \Xh{1}= \pr(K \oplus \C).
\]
Define similarly  $ \Fh{1}= \pr( K^{\mcf} \oplus \C).$
\end{defn}
Denote the triple covering map
\be
\tb{c}:\Fh{1}  \to \F{1}, \n
\ee
which is the standard branched triple cover when restricted to the fibers;
\begin{align}
\cp{1} &\to \cp{1},\n\\
[x,y] & \mapsto [x^3,y^3].\n
\end{align}
It is branched at the two points $0 = [0,1]$ and $\infty = [1,0]$.
By construction, $\tb{c}$ descends to the triple covering map
$$
\tb{c}:\Xh{1}  \to \X{1}.
$$

Let $\Ih{1}=\tb{c}^*\I{1}$ be the pulled back ideal, and define the new differential system
$$(\Xh{1},\,\Ih{1}).$$
It is clear that the present construction has the naturality so that
the triple cover of a minimal Lagrangian surface $\Sigmah\to\Sigma$ admits 
a unique lift to $\Xh{1}$
as an integral surface of $\Ih{1}$.
\begin{defn}
For each $2\leq k \leq \infty$,  define the \tb{triple cover} of the prolongations
\begin{align*}
 \Xh{k}&:=\tb{c}^*(\X{k}), \\
\Fh{k}&:= \tb{c}^*(\F{k}),
\end{align*}
such that we have the commutative diagram:\one
\\
\centerline{
\xymatrix{
 \Xh{k} \ar[d]^{\hat{\pi}_{k,1}} \ar[r]^{\tb{c}_k} & \X{k} \ar[d]^{\pi_{k,1}}  \\
 \Xh{1}  \ar[r]^{\tb{c}} & \X{1}
}}\\

\one\noi  (and similarly for $\Fh{k}, \Fh{1}$, etc). 
Set the ideals $\Ih{k}:=\tb{c}_k^*(\I{k})$.
Denote the pulled-back canonical bundles by
\begin{align*} 
\pi_{k,0}^*K&:=K\to \X{k},\\
\tb{c}_k^*K&:=\hat{K} \to \Xh{k}.\n
\end{align*}
\end{defn}
The following proposition summarizes the construction so far.
\begin{prop}
Let $\x:\Sigma\hook X$ be an immersed integral surface of the differential system for minimal Lagrangian surfaces.
Let $$\x^{(k)}:\Sigma\hook\X{k},\,1\leq k \leq\infty,$$ 
be the prolongation of $\x$.
Let $\nu:\Sigmah\to\Sigma$ be the triple cover defined by the Hopf differential of $\x$, Definition \ref{defn:doublecover}.
There exists a lift $\xh^{(1)}:\Sigmah\hook\Xh{1}$ and the associated sequence of prolongations
\[ \xh^{(k)}:\Sigmah\hook\Xh{k},\quad 2\leq k \leq\infty,
\]
such that
\begin{enumerate}[\qquad a)]
\item  each $\xh^{(k)}$ is integral to $\Ih{k}$,
\item  $\x^{(k)}\circ \nu =\tb{c}_k\circ\xh^{(k)}$.
\end{enumerate}
The lift $\,\xh^{(1)}$ and its prolongation sequence $\{\,\xh^{(k)}\,\}$ are uniquely determined by these  properties.
\end{prop}

\subsection{Balanced coordinates}\label{sec:balanced}
As is often the case with integrable differential equations, the infinite prolongation of the minimal Lagrangian system supports a preferred ring of functions 
called \emph{balanced coordinates}.\ftmark
\fttext{They form a coordinate system when restricted to a fiber 
of the projection $\xinfh_{**}\to\Xh{1}.$} 
   We will find that many of the higher-order objects we are interested in, such as higher-order Jacobi fields and conservation laws, admit the weighted homogeneous expressions
   in terms of these functions.  

\subb{Set up}\label{sec:balancedmoti}
Recall the open subsets $\X{k}_0, \X{k}_{\infty},$ $1\leq k \leq \infty$, from \eqref{eq:X10} and below.
Let $\X{1}_*\subset \X{1}$ be their intersection,
\be
\X{1}_*:=\X{1}_0\cap\X{1}_{\infty}=\{ \, \tn{E}\in\X{1}\;\vert\; \;h_3\vert_E\ne0,\infty\,\}.\n\ee
Inductively define the sequence of open subsets
$\X{k+1}_*:=\pi_{k+1,k}^{-1}(\X{k}_*)\subset \X{k+1}.$
Let
\be
\X{k}_{**}:=\X{k}_*\cap \{ \, h_j \ne \infty,\; 4\leq j \leq k+2\, \}\subset \X{k}_*, \quad k\geq 2.\n
\ee
Denote the corresponding open subsets on the triple cover by  
$$\Xh{k}_0, \Xh{k}_{\infty}, \Xh{k}_*, \Xh{k}_{**}, \;\tn{etc}.
$$
The balanced coordinates to be constructed will be defined on $\Xh{\infty}_{**}.$

Let  $\{\F{k}_0, \F{k}_{\infty}, \F{k}_*, \F{k}_{**}, \Fh{k}_0, \Fh{k}_{\infty}, \Fh{k}_*, \Fh{k}_{**} \}$ 
denote the associated open sets.
  
\subb{Definition}\label{sec:balanceddefi}
With this preparation, we introduce a preferred set of functions on $\xinfh_{**}$ 
which are adapted for the prolongation structure of the minimal Lagrangian system.
\begin{defn}\label{defn:balancedz}
A sequence of \tb{balanced coordinates}
$z_j : \Xh{\infty}_{**} \to \C,   j \geq 4,$ are defined by
$$ z_j:=h_3^{-\frac{j}{3}}h_j.
$$
\end{defn}
It is clear that these functions, originally defined on $\Fh{\infty}_{**}$, are invariant under the action of the structure group $\SO(2)$. They are well defined on $\xinfh_{**}$.

\subb{Structure equation}\label{sec:balancedstrt}
We record the structure equation for the balanced coordinates.
Set
\be\left\{
\begin{array}{rl}
\tn{r}&:=|h_3|=(h_3\hb_3)^{\frac{1}{2}}, \n \\
\omega&:=h_3^{\frac{1}{3}} \xi, \n  \\
\zeta_j&:=h_3^{-\frac{j}{3}}\theta_j,\qquad\quad j\geq 0,  \n  \\
\hat{T}_j &:=h_3^{-\frac{j-1}{3}}T_j, \;\;\quad\quad j\geq 3.\n\\
\end{array}\right.
\ee
They are all invariant under the action of the structure group $\SO(2)$ 
and hence well defined on $\Xh{\infty}_{**}$.  
This enables to express the ideal $\iinfh$ on $\Xh{\infty}_{**}$  by
$$\iinfh=\langle \theta_0, \zeta_j, \zetab_j \rangle,$$
and the generators are defined  on $\Xh{\infty}_{**}$ itself.

\two
Note the following structure equations:
\begin{align}\label{eq:zetastruct}  
\ed \tn{r}&\equiv\frac{\tn{r}}{2}(z_4\omega+\zb_4\omb),  \\
\ed z_j&\equiv \left(z_{j+1}-\frac{j}{3} z_4 z_j \right) \omega
+ \hat{T}_{j} \tn{r}^{-\frac{2}{3}}\omb  \;\mod \iinfh, \quad{\rm for} \; j \geq 4,\n\\
 \hat{T}_{j+1}&= \del_{\omega} \hat{T}_{j}+\frac{j-1}{3}z_4  \hat{T}_j
+\frac{j}{2}(\gamma^2- 2 \tn{r}^2)z_j, \quad{\rm for} \; j \geq 3.\n
\end{align}
Here we set $z_3=1$ for convenience, and denote by $\del_{\omega}$ the operator 
$$\del_{\omega}:=h_3^{-\frac{1}{3}}\delx.$$

From these formulas, we note an important property of $\hat{T}_j$.
\begin{defn}\label{defn:spectralweights}
The \tb{spectral  weights} of the balanced coordinates are
\be\label{eq:spectralweight}
 \tn{weight}(z_j)=j-3, \quad  \tn{weight}(\zb_j)=-(j-3).
\ee
Set $\tn{weight}(\tn{r})=0.$
Assign the weights for the 1-forms
$$ \tn{weight}(\omega)=-1, \quad  \tn{weight}(\ol{\omega})=+1.
$$
\end{defn}
\begin{lem}\label{lem:Tjh}
By definition,
\be
\hat{T}_j \in \C[\tn{r}^2,z_4,z_5,z_6,\ldots].   \n
\ee
It is weighted homogeneous of weight $j-4$ under the spectral weight \eqref{eq:spectralweight}.
\end{lem}
\begin{proof}
From the identities
\begin{align}
\del_{\omega} z_j&=z_{j+1}-\frac{j}{3} z_4 z_j, \n\\
\del_{\omega} ( \tn{r}^2 )&= \tn{r}^2 z_4,\n
\end{align}
the operator $\del_{\omega}$ increases the spectral weight by +1 when acting on
$\C[\tn{r}^2,z_4,z_5,z_6,\ldots]$. Note the initial term
$\hat{T}_4=\frac{3}{2}(\gamma^2-2\tn{r}^2)$ is of weight $\,0$.
The rest follows from the inductive formula for $\hat{T}_j$ in \eqref{eq:zetastruct}.
\end{proof}

\sub{Order vs. spectral weight}
The prolongation variable $h_3$ represents the second fundamental form of the minimal Lagrangian surface, and hence it is a second order object.
Consequently, the jet order of the balanced coordinate $z_j$ is $j-1$.

However, for the sake of convenience, 
we re-define the \tb{order} of the functions $z_j, \zb_j$ as follows:
\be\label{table:order}
\begin{array}{r|r|r}
&\tn{order}&\tn{spectral weight}\\
\hline
z_j & j & j-3\\
\zb_j& j &-(j-3)
\end{array}\ee

\section{Two lemmas}\label{sec:prolonglemmas}    
We record two useful lemmas regarding the $\delxb$-equation on $\xinfh$.
Among other things, they will be applied to give a classification of the higher-order Jacobi fields in \S\ref{sec:Jacobifields}.

Lemma~\ref{lem:lemma5.4'} is a variant of the fact that the minimal Lagrangian system is not Darboux integrable at any order.
Lemma~\ref{lem:delbpoly0} is a rigid property of the polynomial ring $\C[z_4,z_5,z_6, \, ... \, ]$ under the $\delxb$-operator.

\subsection{Lemma \ref{lem:lemma5.4'} }
\begin{lem}\label{lem:lemma5.4'}
Let $f: U\subset \xinfh\to\C$ be a scalar function on an open subset $U\subset\xinfh$ such that
$$\delxb f= c h_3^{-\frac{1}{3}}$$
for a constant $c$.  From the structure equation, this implies that for some $k>0$,
\be\label{eq:Lemma5.4'}
\ed f    \equiv c h_3^{-\frac{1}{3}}\xib \mod\;  \xi,   \theta_0,  \theta_1,  \thetab_1, \theta_2,\thetab_2,
\theta_3, \theta_4,\,...\,\theta_{k},
\ee
i.e., such $f$ does not depend on any of the conjugate variables $\hb_j$, $j\geq 3.$
Then $c=0$ necessarily and $f$ is a constant function.
\end{lem}
\begin{cor}\label{cor:lemma5.4}
Let $f: U\subset \xinfh\to\C$ be a scalar function on an open subset $U\subset\xinfh$ such that
\be\label{eq:Lemma5.4}
\delxb f =0.
\ee
Then $f$ is a constant function.
\end{cor}
The corollary states that the infinitely prolonged minimal Lagrangian differential system is not Darboux integrable,
see \cite{Bryant1995b} for a discussion of Darboux integrability (in the hyperbolic case). Roughly, this means that no matter how many times one differentiates, the minimal Lagrangian system under consideration, $\gamma^2\ne 0$, does not admit a  Weierstra\ss\, type of holomorphic representation formula. This is in contrast with the case $\gamma^2=0$ where the minimal Lagrangian system is   equivalent to the holomorphic differential system for complex curves in $\C^2.$

\two
A proof of Lem.\ref{lem:lemma5.4'} can be obtained by a direct adaptation of the proof of the corresponding lemma for the differential system for constant mean curvature surfaces given in \cite{Wang2013}. Although the analysis goes by straightforward computations, it is a little involved and, 
in order to avoid repetition, let us content ourselves with a brief description of the relevant ideas.

We shall apply  induction  on $k$.
Assume $k\geq 5$. The case $k\leq 4$ can be checked by a direct computation.

Let $\mathtt{I}$ be the Pfaffian system generated by
\[ \mathtt{I}=\langle \, \thetab_2, \thetab_1,\,\theta_0,\,\xi,\,\theta_1,\,\theta_2,\, ... \,\theta_{k}\,\rangle.
\]
Our claim is that there is no nonzero closed 1-form $\alpha$ of the form
$$\ed f=\alpha= c h_3^{-\frac{1}{3}}\xib + \btheta, \quad \btheta\in\mathtt{I}.
$$
 
Denote  a 1-form in $\texttt{I}$ by
$$\btheta=a_{\bar{2}}\thetab_2+a_{\bar{1}}\thetab_1+a_0\theta_0
+a_{\xi}\xi+a_1\theta_1+\sum_{j=2}^{k}a_j\theta_j.$$
From the equation (which follows from the identity $\ed^2=0$) 
$$\ed\alpha\equiv 0\mod\mathtt{I},$$ 
collect $\thetab_3\w\xib$-terms and one gets
\[ a_{\bar{2}}=-\sum_{j=4}^ka_j T^{\bar{3}}_j.
\]

Let $\mathtt{I}^{(\bar{1})}$ be the Pfaffian system generated by
\[ \mathtt{I}^{(\bar{1})}=\langle \,\ol{\btheta}_1,  \btheta_0,\,\xi,\,\btheta_1,\,\btheta_2,\, ... \,\btheta_{k}\,\rangle,
\]
where we now set
\begin{align}
\ol{\btheta}_{1}&=\thetab_{1}, \btheta_{0}=\theta_0, \;\btheta_1=\theta_1, \;\btheta_2=\theta_2,  \;\btheta_3=\theta_3, \n\\
\btheta_j&=\theta_j-T^{\bar{3}}_j\thetab_2, \quad \tn{for}\;j\geq 4.\n
\end{align}
Denote  a (new) 1-form in $\mathtt{I}^{(\bar{1})}$ by
$$\btheta=a_{\bar{1}}\ol{\btheta}_{1}+a_0\btheta_0
+a_{\xi}\xi+a_1\btheta_1+\sum_{j=2}^{k}a_j\btheta_j.$$
Now our claim is that there is no nonzero closed 1-form $\alpha$ of the form
$$ \alpha= c h_3^{-\frac{1}{3}}\xib + \btheta, \quad \btheta\in\mathtt{I}^{(\bar{1})}.
$$

Repeat the similar computations and, by induction argument,  solve for the sequence of coefficients 
$\{\, a_{\bar{1}}, a_0, a_{\xi}, a_1, a_2, \, ... \, \}$.
Continuing this process,  one arrives at the following  normal form for the closed 1-form $\alpha$:
$$\ed f=  \alpha= c  \left(h_3^{-\frac{1}{3}}\xib +f^{\xi}\xi+\sum_{j=-2}^{k}f^j\theta_j\right),
\qquad ( \theta_{-2}=\ol{\theta}_{2}, \; \tn{etc}),
$$
for the given constant $c$, where $$f^j\in\mco(k-1)\quad\tn{for}\;\; j\geq 4.$$

\two
\begin{itemize}
\item Suppose $c=0$. Then $\ed f =0$ and $f$ is a constant.

\item Suppose $c\ne 0$. We show that this leads to a contradiction.
Applying  the commutation relation \eqref{eq:lxib},
$$[E_k, \delxb]=\sum_{j\geq k+1}T_j^k E_j,$$ 
to the given function $f$,
one gets
$$ \delxb(f^k)= E_k(ch_3^{-\frac{1}{3}})=0.
$$
By the induction hypothesis, this implies that $f^k$ is a constant multiple of $h_3^{-\frac{k}{3}}$.

\begin{itemize}
\item  Suppose $f^k=0$. Then $f\in\mco(k-1)$ and, again by the induction hypothesis, $f\equiv \tn{constant}$. Hence $c=0$, a contradiction.

\item Suppose $f^k\ne 0$.
Applying  the commutation relation 
$$ [E_{k-1}, \delxb]=\sum_{j\geq k}T_j^{k-1} E_j$$ 
to the given function $f$, one gets (since $k\geq 5$)
$$ -\delxb(f^{k-1})=T^{k-1}_k f^k.
$$
It is easily checked that, using the induction hypothesis, this forces $f^k=0$, a contradiction.
\end{itemize}\end{itemize}

\subsection{Lemma \ref{lem:delbpoly0}}
Consider the polynomial rings in the balanced coordinates,
$$\mcr:=\C[z_4, z_5, z_6, \, ... \, ],  \quad \ol{\mcr}:=\C[\zb_4, \zb_5, \zb_6, \, ... \, ].
$$
Recall the spectral weights in Defn.\ref{defn:spectralweights}.
Define the associated sequences of polynomial vector spaces filtered by the spectral weight 
as follows: 
\be
\begin{cases}
\quad P_d&=\{ \tn{weighted homogeneous polynomials of degree $d\geq 0$ in  $\mcr$} \}, \n\\
\quad \mcp_d&=\oplus_{i=0}^{d} P_i\subset\mcr, \n\\
\quad \mcp_d(k)&=\mcp_d \cap \mco(k)\subset\mcr\cap \mco(k), \n\\
\quad Q_d&=P_d\oplus (h_3\hb_3) P_d, \n\\ 
\quad \mcq_d&= \oplus_{i=0}^{d} Q_{i},\n\\
\quad  \mcq_d(k)&=\mcq_d \cap \mco(k).\n
\end{cases}
\ee

\one
The following lemma, which is an application of Lem.\ref{lem:lemma5.4'},
records a characteristic rigidity property of the subspace $\mcp_d(k)$ 
under the differential operator $h_3^{\frac{1}{3}}\delxb$.  
\begin{lem}\label{lem:delbpoly0}
Let $v\in\mco(k), k\geq 4$.  
Suppose
$$h_3^{\frac{1}{3}}\delxb v\in\mcq_{d}(k). 
$$
Then $$v\in\mcp_{d+1}(k).$$
\end{lem}
\begin{cor}\label{cor:delbpoly}
Let $u\in\mco(k), k\geq 4$.  Let $u^{k}=E_{k} (u)=\frac{\del u}{\del h_{k }}$.
Suppose
$$h_3^{\frac{1}{3}}\delxb (h_3^{\frac{k}{3}} u^{k }) \in\mcq_{d}(k). 
$$
Then $h_3^{\frac{k}{3}}u^k\in\mcp_{d+1}(k)$, and hence
$$u\in\mcp_{d+(k-2)}(k)\mod \mco(k-1).$$
\end{cor}
\begin{proof}
Substitute $v=h_3^{\frac{k}{3}} u^{k }$ from Lem.\ref{lem:delbpoly0}.
\end{proof}
\noi\emph{Proof of Lem.~ \ref{lem:delbpoly0}}.\, 
Consider the case $k=4.$ For functions in $\mco(4)$, we have the commutation relation
$[E_4,\delxb]=0.$ Hence by applying $E_4$ repeatedly to $\delxb v$, we get
$$h_3^{\frac{1}{3}}\delxb E^{m}_4(v)=0$$
for some $m\leq d+1$.
By Cor~.\ref{cor:lemma5.4}, $v$ is a polynomial in $z_4$ and one may write
$$v= v_m z_4^m+v_{m-1}z_4^{m-1}+\, ... \, +v_1 z_4+v_0,
$$
where  $v_j\in\mco(3)$ with $v_m$ being a constant.
Substitute this to the given equation for $h_3^{\frac{1}{3}}\delxb v$, 
and one gets the recursive equations
$$h_3^{\frac{1}{3}}\delxb v_j+(j+1)v_{j+1}\frac{3}{2}R
=c'_j\gamma^2+c_j^{"}h_3\hb_3, \quad j=m, m-1, \, ... \,,$$
for constants $c'_j,c_j^{"}$ (set $v_{m+1}=0$).
Since $v_m$ is a constant and 
$$h_3^{\frac{1}{3}}\delxb z_4=\frac{3}{2}R=\frac{3}{2}(\gamma^2-2h_3\hb_3),$$
whereas $v_j\in\mco(3)$,
an inductive argument using  Lem.\ref{lem:lemma5.4'} for $j$ from $m-1$ to $0$ shows that
all the coefficients $v_j$ must be constant.  This shows $$v\in\mcp_{d+1}(4).$$

Applying the induction argument, suppose the claim is true up to $\mco(k-1)$.
Let $v\in\mco(k)$.  For functions in $\mco(k)$, we have the commutation relation
$[E_k,\delxb]=0.$ Hence, similarly as above,  $v$ is a polynomial in $z_k$  
and one may write
$$v= v_m z_k^m+v_{m-1}z_k^{m-1}+\, ... \, +v_1 z_k+v_0,
$$
where $v_j\in\mco(k-1)$, $m(k-3)\leq d+1$, and $v_m$ being a constant.
Substitute this to the given equation for $h_3^{\frac{1}{3}}\delxb v$, 
and one gets the recursive equations
$$h_3^{\frac{1}{3}}\delxb v_j+(j+1)v_{j+1}\hat{T}_k\in \mcq_{d-j(k-3)}(k-1),
\quad j=m, m-1, \, ... \,.$$
By Lem.\ref{lem:Tjh}, $\hat{T}_k\in  \mcq_{k-4}(k-1)$.
It follows by the similar inductive argument as above, with decreasing $j$ from $m-1$ to $0$,
that  $$v_j\in \mcp_{d-j(k-3)+1}(k-1), \quad  m-1\geq j\geq 0.$$
This shows $v\in\mcp_{d+1}(k).$  \hfill $\square$

\section{Jacobi fields}\label{sec:Jacobifields}
From the general theory of differential equations, \cite{Vinogradov1989},
a generating function of  symmetry of a differential equation
is characterized as an element in the kernel of its linearization.
For   the minimal Lagrangian system, it turns out that
the generating functions of symmetries are Jacobi fields.

Since the normal bundle of a Lagrangian surface is canonically isomorphic to the cotangent bundle,
the corresponding Jacobi operator (i.e., the defining equation for Jacobi fields)
reduces to the second order operator
$$\delx\delxb+\frac{3}{2}\gamma^2.$$
In particular, Jacobi fields correspond to the eigenfunctions of 
 Laplacian of the induced metric. 

In this section, we give a complete classification of  the (pseudo) Jacobi fields
for the minimal Lagrangian system
based on the results of \S\ref{sec:prolonglemmas}.
An infinite sequence of higher-order (pseudo) Jacobi fields will be constructed 
in the next section by
the similar  recursion procedure as discussed in \S\ref{sec:recursionrel}.

\two
In \S\ref{sec:JacobiSymmetry}, 
we show that  a Jacobi field is a generating function of symmetry
for the minimal Lagrangian system.
In \S\ref{sec:splitting}, 
we prove a splitting theorem that a Jacobi field 
decomposes as the sum of a classical Jacobi field 
and a higher-order Jacobi field. 
Combining this with the results from \cite{Fox2011,Fox2012}, 
we obtain a complete classification of Jacobi fields in \S\ref{sec:classification}. 
We find that a higher-order ($\geq5$) Jacobi field exists at each odd order $5,$ or $1$ mod $6.$

The similar analysis applies to the 
classification of  pseudo-Jacobi fields.
A pseudo-Jacobi field  corresponds to a symmetry of 
the elliptic Tzitzeica equation underlying the minimal Lagrangian system.
We find that a higher-order ($\geq 4$)  pseudo-Jacobi field exists 
at each even order $4,$ or $2$ mod $6.$

\subsection{Jacobi fields and pseudo-Jacobi fields}\label{sec:JacobiSymmetry}
\subb{Definition}
Motivated by Defns.\ref{defn:classicalJacobifield},  \ref{defn:classicalpseudoJacobifield},
we give a general definition of (pseudo) Jacobi fields on $\xinfh.$\ftmark\fttext{It turns out that the 
infinite sequence of higher-order (pseudo) Jacobi fields belong  to
$\mcr\oplus\ol{\mcr}\subset C^{\infty}(\xinfh_{**}).$}
\begin{defn}\label{defn:Jacobifield}
A scalar function $A:\xinfh\to\C $ is a \tb{Jacobi field} if it satisfies the linear Jacobi equation
\be\label{eq:Jacobieq}
\mce(A):=\delx\delxb A+\frac{3}{2} \gamma^2A=0.
\ee
The $\C$-vector space of Jacobi fields is denoted by  $\mfj^{(\infty)}$.

A scalar function $P:\xinfh\to\C$ is a \tb{pseudo-Jacobi field} if it satisfies the linear pseudo-Jacobi equation
\be\label{eq:Jacobieq'}
\mce'(P):=\delx\delxb P+\frac{1}{2} ( \gamma^2+4h_3\hb_3)P=0.
\ee
The $\C$-vector space of pseudo-Jacobi fields is denoted by  $\mfj'^{(\infty)}$.

A \tb{higher-order Jacobi field} is a Jacobi field in the ring of the balanced coordinates $\mcr\oplus\ol{\mcr}$. The subspace of higher-order Jacobi fields is denoted by $\mfj^{(\infty)}_h$.

Let $\mfj^{(k)}\subset\mfj^{(\infty)}$ be the subspace of \tb{Jacobi fields of order $\leq k+2$}  defined on $\Xh{k}$. Let $\mfj^{(k)}_h=\mfj^{(k)}\cap\mfj^{(\infty)}_h.$
The space of \tb{Jacobi fields of order $k+2$} is defined as the quotient space
\begin{align}
\mfj^k&=\mfj^{(k)}/\mfj^{(k-1)},\quad k\geq 1,  \n  \\
\mfj^0&=\mfj^{(0)}=\{\tn{\,classical Jacobi fields }\}.\n
\end{align}

The corresponding subspaces of pseudo-Jacobi fields are denoted  by
$\mfj'^{(\infty)}_h$,  $\mfj'^{(k)}$, $\mfj'^{(k)}_h$, $\mfj'^k,\mfj'^{(0)}$.
\end{defn}
\subb{Stability lemma}
Before we proceed to the analysis of Jacobi fields and symmetries, 
let us record the useful formulas related to the stability of the ring of balanced coordinates under the Jacobi operators $\mce,\mce'$.
\begin{lem}\label{lem:mcez}
Consider  the polynomial ring $\mcr=\C[z_4, z_5, \, ... \, ]$ of balanced coordinates.
We have,
\begin{align}
\mce(z_j)&=\hat{T}_{j+1}-\frac{j}{3}(z_j \hat{T}_4+z_4 \hat{T}_j)+\frac{3}{2} \gamma^2z_j, \n\\
\mce'(z_j)&=\hat{T}_{j+1}-\frac{j}{3}(z_j \hat{T}_4+z_4 \hat{T}_j)+\frac{1}{2}(\gamma^2+4\tn{r}^2) z_j.\n
\end{align}
It follows that
$$\mce(\mcr), \mce'(\mcr)\subset \mcr\oplus (h_3\hb_3)\mcr.
$$
In particular, the operators $\mce, \mce'$ preserve the spectral weight when acting on $\mcr$.
\end{lem}
Note $\hat{T}_4=\frac{3}{2}(\gamma^2-2 \tn{r}^2).$
 
\subb{Jacobi field  and symmetry}
It was observed in \S\ref{sec:classicallaws} that there exists an isomorphism between the classical Jacobi fields and the classical symmetries. The higher-order analogue of this isomorphism is true, and a Jacobi field uniquely determines a (vertical) symmetry vector field of $(\xinfh, \iinfh)$.

\two
Consider the coframe of $\Fh{\infty}$, 
\[\{\,\rho,\,\xi,\,\xib,\,\theta_0,\,\theta_1,\,\thetab_1,\,\theta_2,\, \thetab_2,\,... \,\},\]
and its dual frame
\[\{\,E_\rho,\,E_{\xi},\,E_{\xib},\,E_0,\,E_1,\,E_{\bar{1}},\,E_2,\, E_{\bar{2}},\,... \,\}.\]
By a vector field on $\xinfh$, we mean a vector field (derivation) on $\Fh{\infty}$ of the form
$$
V=V_{\xi}E_{\xi} +V_{\xib}E_{\xib}+V_0 E_0+\sum_{j=1}^{\infty} ( V_j E_j + V_{\bar{j}}E_{\bar{j}})$$
which is invariant under the action of the structure group $\SO(2)$ of the principal bundle $\Fh{\infty}\to\xinfh$. We denote the set of vector fields on $\xinfh$ by $H^0(T\xinfh)$.
\begin{defn}
A vector field  $V\in H^0(T\xinfh)$  is a \tb{symmetry} of the differential system $(\xinfh,\Ih{\infty})$
 if the formal Lie derivative $ \mcl_V$ preserves the ideal $\Ih{\infty}$,
\[  \mcl_V  \Ih{\infty} \subset \Ih{\infty}.\]
The $\C$-algebra of symmetry vector fields is denoted by $\mathfrak{S}$.

A symmetry $V\in\mathfrak{S}$ is \tb{vertical} when $V_{\xi}=V_{\xib}=0$ and it has no (horizontal) $E_{\xi}, \,E_{\xib}$ components. The subspace of vertical symmetries is denoted by $\mathfrak{S}_v$.
\end{defn}
 
\one
We wish to give an analytic characterization of symmetry. Consider a vertical symmetry
\be\label{eq:symvf}
V=V_0 E_0+\sum_{j=1}^{\infty} ( V_j E_j + V_{\bar{j}}E_{\bar{j}}).
\ee
We claim that $V_0$ is necessarily a Jacobi field,
and that it is the generating function of symmetry in the sense that
$V_j, V_{\bar{j}}$'s are determined by $V_0$ and its successive derivatives.
 
\two
\textbf{Step 0}.
The condition that the Lie derivative $\mcl_V\theta_0\equiv 0\mod\Ih{\infty}$ shows that
\begin{align}
\ed V_0-V\lrcorner \left(\frac{1}{2}(\theta_1\w\xi+\thetab_1\w\xib)\right)&\equiv
\ed V_0 - \frac{1}{2}(V_1 \xi+V_{\bar{1}} \xib) \n\\
&\equiv 0 \mod \iinfh.\n
\end{align}
One gets
\be\label{eq:symstep0}
\ed V_0\equiv \frac{1}{2}(V_1 \xi+V_{\bar{1}} \xib)\mod \iinfh.
\ee

\textbf{Step 1}.
By a similar computation, the condition that the Lie derivatives $\mcl_V\theta_1, \, \mcl_V\thetab_1\equiv 0\mod\I{\infty}$ shows that
\be\label{eq:symstep1}\begin{array}{rcrll}
 \ed V_1 + \im V_1\rho&\equiv&
  V_2 \xi&-3\gamma^2 V_0\xib,& \\
\ed V_{\bar{1}} - \im V_{\bar{1}}\rho&\equiv&
V_{\bar{2}}\xib&-3\gamma^2 V_0\xi,&\mod\;\iinfh.
\end{array}
\ee
Note that Eq.\eqref{eq:symstep0} and Eq.\eqref{eq:symstep1} imply that $V_0$ is a Jacobi field.

\textbf{Step j}.
The rest of the coefficients $V_j, V_{\ol{j}}, j\geq 2,$ are similarly determined by $V_0$ and its successive derivatives by the condition that the Lie derivatives 
$\mcl_V\theta_j, \, \mcl_V\thetab_j\equiv 0\mod\Ih{\infty}$. 
These equations imply that
\be\label{eq:symstepj}\begin{array}{rcll}
 \ed V_j +j \im V_j\rho&\equiv&
 \left( V_{j+1} +\{ V_j, V_{\bar{j}}, \, ... \, V_0\} \right) \xi+\{ V_{j-1}, V_{\ol{j-1}}, \, ... \, V_0\}\xib, \\
 \ed V_{\bar{j}} -j \im V_{\bar{j}}\rho&\equiv&
 \left(V_{\ol{j+1}} +\{ V_j, V_{\bar{j}}, \, ... \, V_0\} \right) \xib+\{ V_{j-1}, V_{\ol{j-1}}, \, ... \, V_0\}\xi,
\mod\;\iinfh.&\\
\end{array}
\ee
Here $\{ V_{j}, V_{\ol{j}}, \, ... \, V_0\}$, etc, is a generic notation for an expression
which is linear in $\{ V_k, V_{\bar{k}}\}_{k=0}^{j}$ with coefficients in the ring $\C[h_3, \hb_3,h_4,\hb_4, \, ... \, ]$.
 
\two
In fact, the following converse of this analysis is true.
\begin{prop}\label{prop:symmetry}
The generating function of a vertical symmetry of the differential system for minimal Lagrangian surfaces 
is a Jacobi field.
Conversely, a Jacobi field $A$ uniquely determines a vertical symmetry $V$ 
of the form \eqref{eq:symvf} with the generating function $V_0=A$.
As a consequence, there exists a canonical isomorphism
\[ \mathfrak{J}^{(\infty)}\simeq \mathfrak{S}_v.\]
\end{prop}
\begin{proof}
The compatibility of the recursive defining equations \eqref{eq:symstep0}, \eqref{eq:symstep1}, and \eqref{eq:symstepj} can be checked from the formula for $T_j$ and its differential consequences.
We omit the details.
\end{proof}

\subb{Interpretation of  pseudo-Jacobi field}\label{sec:interpseudo}
The pseudo-Jacobi fields correspond  to the symmetries of 
the elliptic Tzitzeica equation underlying the minimal Lagrangian system.
The symmetries of the elliptic Tzitzeica equation 
are classified in \cite{Fox2011,Fox2012}.

Away from the umbilic divisor on a minimal Lagrangian surface, 
take a local holomorphic coordinate $z$ such that
$$\ed z =h_3^{\frac{1}{3}}\xi,$$
and the Hopf differential is normalized to $\ff=(\ed z)^3.$
Set accordingly,
$$\xi=e^{\frac{u}{2}}\ed z, \quad h_3=e^{-\frac{3}{2}u}$$
for a real scalar function $u$.
The connection 1-form $\rho$ is given by
$\rho=\frac{\im}{2}(u_z\ed z -u_{\ol{z}}\ed\ol{z}),$ 
and it follows that the curvature $R$ is,
$$R=-2 e^{-u} u_{z\ol{z}}.$$
The compatibility equation \eqref{eq:Gauss} then translates to the elliptic Tzitzeica equation
\be\label{eq:Tzitzeica}
u_{z\ol{z}}+\frac{1}{2}\left(\gamma^2 e^u-2 e^{-2u}\right)=0.
\ee

On the other hand, consider the pseudo-Jacobi operator \eqref{eq:Jacobieq'}.
From $\ed z=h_3^{\frac{1}{3}}\xi$, one gets $\del_z=h_3^{-\frac{1}{3}}\delx$. Hence
$$\delx\delxb=e^{-u} \del_z\del_{\ol{z}},$$
and the pseudo-Jacobi operator  translates to
$$\mce'=e^{-u}\left(\del_z\del_{\ol{z}}+\frac{1}{2}(\gamma^2 e^u+4 e^{-2u})\right).$$
Up to scaling by $e^{-u}$, this is the linearization of \eqref{eq:Tzitzeica}
and the claim follows.

\subb{Examples}
\begin{exam}\label{exam:initialdata}
A direct computation show that
$$z_4$$ is a pseudo-Jacobi field (of order 4 and degree 1), and
$$z_5-\frac{5}{3}z_4^2$$ is a Jacobi field  (of order 5 and degree 2).
\end{exam}
This example provides a hint for the existence of  higher-order (pseudo) Jacobi fields.

\subsection{Splitting theorem}\label{sec:splitting}
For a  nonlinear  differential equation, it is a stringent condition to admit 
a higher-order symmetry, not to mention an infinite number of them. 
The corresponding Jacobi equation  would imply
a variety of compatibility conditions for the given differential equation, 
and generically one expects that the space of Jacobi fields is trivial.

In this section,
we explore the constraints given by the Jacobi equation 
of the minimal Lagrangian system
by a repeated application of Lem.\ref{lem:lemma5.4'}, Lem.\ref{lem:delbpoly0}, 
and Cor.\ref{cor:delbpoly}. 
As a result, we obtain a rough normal form for Jacobi fields and this implies  that 
the space of Jacobi fields splits into the direct sum of the classical Jacobi fields and the higher-order Jacobi fields. 

A similar splitting theorem holds for the  pseudo-Jacobi fields.

\subb{Symbol lemma}
We start with a lemma on the normal form for the highest order term of a Jacobi field.
\begin{lem}\label{lem:Jacobinormal}
Let $A\in \mfj^{(k)}\subset\mco(k+2)$ be a (pseudo) Jacobi field. Suppose $A^{k+2}=E_{k+2}(A)\ne 0$.
Then $A$ is at most linear in the highest order variable $z_{k+2}=h_{3}^{-\frac{k+2}{3}}h_{k+2}$, and, 
up to constant scale, it admits the normal form
$$ A=z_{k+2}+\mco(k+1).
$$
\end{lem}
\begin{proof}
Since $A\in\mco(k+2)$, we have
\begin{align}
A_{\xi}&\equiv h_{k+3} A^{k+2} \mod \mco(k+2), \n\\
A_{\xi,\xib}&\equiv h_{k+3} \delxb(A^{k+2})\mod \mco(k+2),\n \\
&\equiv 0\mod\mco(k+2), \quad \tn{for \;$\mce(A)=0$\, (\tn{or}\; $\mce'(A)=0$).}\n
\end{align}
This forces $\delxb(A^{k+2})=0$. By Cor.\ref{cor:lemma5.4}, $A^{k+2}$ is a constant multiple of $h_3^{-\frac{k+2}{3}}$.
\end{proof}

\subb{Splitting theorem}
We refine Lem.\ref{lem:Jacobinormal} 
to the splitting theorem for (pseudo) Jacobi fields with the help of Cor.\ref{cor:delbpoly}.

Recall $\mcr=\C[z_4, z_5, z_6,\,...  \, ]$, $\ol{\mcr}=\C[\zb_4, \zb_5, \zb_6,\,...  \, ].$
\begin{prop}\label{prop:splitting}
The space of Jacobi fields splits into the direct sum
$$ \mathfrak{J}^{(\infty)}= \mathfrak{J}^{(\infty)}_h\oplus  \mathfrak{J}^{(0)}
$$
of the higher-order Jacobi fields and the classical Jacobi fields.
By definition,
$$ \mathfrak{J}^{(\infty)}_h \subset \mcr\oplus \ol{\mcr}
$$
and the space of higher-order Jacobi fields is generated by the un-mixed weighted homogeneous polynomial Jacobi fields.

A similar splitting theorem holds for the space of pseudo-Jacobi fields.
\end{prop}
\begin{proof}
Consider first the Jacobi field case.
Let $A\in\mco(k)$ be a Jacobi field for
$k\geq 5$.  By Lem.\ref{lem:Jacobinormal} above, we may set
$$ A=z_{k}+u_{(1)},\quad u_{(1)}\in\mco(k-1).$$
Applying the Jacobi operator, one finds
$$-\mce(z_{k})\equiv h_3^{\frac{1}{3}}\delxb(h_3^{\frac{(k-1)}{3}}u_{(1)}^{k-1})z_{k}\mod\mco(k-1).$$
By Lem.\ref{lem:mcez}, $\mce(z_{k})\in\mcq_{k-3}(k)$.
By Cor.\ref{cor:delbpoly}, one gets
$$ u_{(1)} = p_{(1)}+\mco(k-2)$$
for some $p_{(1)}\in \mcp_{k-3}(k-1).$

Suppose by induction we arrive at the formula
\be\label{eq:Az2kj}
A=z_{k}+p_{(1)}+p_{(2)}+\, ... \, +p_{(j)}+u_{(j+1)},\quad u_{(j+1)}\in\mco(k-(j+1)),
\ee
where each $p_{(i)}\in \mcp_{k-3}(k-i)$ such that
$$ q_{(i)}:=\mce\left(z_{k}+p_{(1)}+p_{(2)}+\, ... \, +p_{(i)}\right)\in\mco(k-i).
$$
By Lem.\ref{lem:mcez}, one finds $q_{(j)}\in\mcq_{k-3}(k-j)$.
Applying the Jacobi operator to the refined normal form \eqref{eq:Az2kj}, one gets
$$ -q_{(j)}\equiv h_3^{\frac{1}{3}}\delxb(h_3^{-\frac{k-(j+1)}{3}}u_{(j+1)}^{k-(j+1)})z_{k-j}\mod\mco(k-(j+1)).$$
By Cor.\ref{cor:delbpoly}, one may write
\begin{align*}
u_{(j+1)}&=p_{(j+1)}+u_{(j+2)}, \\
p_{(j+1)}&\in\mcp_{k-3}(k-(j+1)),\\
u_{(j+2)}&\in\mco(k-(j+2)).
\end{align*}

Continuing this process, we arrive at the normal form
\be\begin{array}{rll}
A&=p+u_{(k-3)}, &u_{(k-3)}\in\mco(3),\n  \\
p&=z_{k}+p_{(1)}+p_{(2)}+\, ... \, +p_{(k-4)},&
\end{array}
\ee
where $p_{(k-4)}\in \mcp_{k-3}(4)$ such that
$$ q_{(k-4)}:=\mce(p)\in\mcq_{k-3}(4).
$$

Since $\mce$ is a real operator, the complex conjugate of this argument implies that
the Jacobi field $A$ decomposes into
$$A =f+g,$$
where  $f$ is an un-mixed pure polynomial in $z_j, \zb_j$, and $g$ is a function on $\Xh{1}$.

Now, the Jacobi operator $\mce$ preserves the spectral weight.
Since $\mce(g)\in\mco(4)\cap\mco(-4)$ is at most linear in $z_4,\zb_4$, 
this implies  
(note $\mce(z_4)=R z_4$), 
\be\label{eq:efeg}
\mce(c'z_4+c''\zb_4)=R(c'z_4+c''\zb_4)=-\mce(g),
\ee
for some constants $c', c"$,
where $c'z_4+c''\zb_4$ denotes the terms of spectral wight $\pm 1$ in the polynomial $f$.

A short computation shows that this forces
$c', c"=0$, and $f\in\mcr\oplus\ol{\mcr}$ is a higher-order Jacobi field.
Hence, Eq.\eqref{eq:efeg} also implies that $\mce(g)=0$ and consequently,
$$h_3 E_3(g), \;\hb_3 E_{\ol{3}}(g)\equiv\tn{const}.$$
It is easily checked from this that $E_3(g), E_{\ol{3}}(g)=0$ and 
$g$ is necessarily a classical Jacobi field defined on $X$.

\two
Consider next the pseudo-Jacobi field case.  By the similar argument as above, 
a  pseudo-Jacobi field $P$ decomposes into
$$P=f+g,
$$
where  $f$ is an un-mixed pure polynomial in $z_j, \zb_j$, and $g$ is a function on $\Xh{1}$.
Since $\mce'(g)\in\mco(4)\cap\mco(-4)$ is at most linear in $z_4,\zb_4$ and $z_4,\zb_4$ are pseudo-Jacobi fields,
it follows that
$$\mce'(f)=0, \quad \tn{and hence }\; \mce'(g)=0.
$$
By the similar argument as above, this implies  that $g$ is necessarily a classical pseudo-Jacobi field.
\end{proof}
 
\subsection{Classification}\label{sec:classification}
Combining the results of \S\S\ref{sec:JacobiSymmetry}, \ref{sec:splitting}, we state a complete classification of (pseudo) Jacobi fields.

The classification consists of the following steps.
We first show, by a direct computation, that 
there exist no even order Jacobi fields in $\mcr\oplus\ol{\mcr}.$
This enables to apply the classification results for elliptic 
Tzitzeica equation in \cite{Fox2011,Fox2012}, and
we find that a nontrivial higher-order Jacobi field exists at order $5$, or $1$ mod $6$ only.
The infinite sequence of higher-order Jacobi fields of the admissible orders will be constructed in \S\ref{sec:formalKilling}.

The classification of pseudo-Jacobi fields follows from
\S\ref{sec:interpseudo}, Prop.\ref{prop:splitting}, and \cite{Fox2011,Fox2012}.
\begin{thm}\label{thm:classifyJacobi}
The infinitely prolonged differential system $(\xinfh,\iinfh)$ for minimal Lagrangian surfaces
in a $2_{\C}$-dimensional, non-flat, complex space form admits an infinite sequence of higher-order (pseudo) Jacobi fields as follows.

\begin{enumerate}[\qquad  a)]
\item
There exists a unique (up to constant scale) nontrivial weighted homogeneous polynomial Jacobi field in $\mcr$ of degree $d\geq 2$ for each 
$$d\equiv 2, \, 4\mod 6,$$
(hence of order $\equiv 5, \,1\mod 6$).
The classical Jacobi fields, 
and these higher-order Jacobi fields and their complex conjugates generate the space of Jacobi fields
$\mathfrak{J}^{(\infty)}.$
\item
There exists a unique (up to constant scale) nontrivial  weighted homogeneous polynomial pseudo-Jacobi field in $\mcr$ of degree $d\geq 2$ for each 
$$d\equiv 1, \, 5\mod 6,$$
(hence of order $\equiv 4, \,2\mod 6$).
The classical pseudo-Jacobi fields, 
and these higher-order pseudo-Jacobi fields and their complex conjugates generate 
the space of pseudo-Jacobi fields
$\mathfrak{J}'^{(\infty)}.$
\end{enumerate}
\end{thm}
 
We present the proof of the theorem in the following two subsections.
\subsubsection{No even order Jacobi fields}\label{sec:noeven}\;
Recall $\mce(\mcr)\subset\mcr\oplus(h_3\hb_3)\mcr.$
Let $F=\frac{3}{2}x_0 z_{2k}+\, ... \, \in\mcr$ be an even order weighted homogeneous polynomial 
Jacobi field of weight $2k-3$ (for a constant $\frac{3}{2}x_0$).
Consider the expansion,
\be\begin{array}{rlll}
F= \frac{3}{2}x_0z_{2k}&+z_{2k-1} (&x_1z_4 &)     \\
                 &+z_{2k-2} (&x_2z_5    &+y_2z_4^2)   \\
                 &+z_{2k-3} (&x_3z_6  &+y_3z_5z_4 +\, ... \, )   \\
                 &+z_{2k-4} (&x_4z_7  &+y_4z_6z_4 +\, ... \, )   \\
& ... && \\
                 &+z_{2k-(i-1)} (&x_{i-1}z_{i+2} &+y_{i-1}z_{i+1}z_4 +\, ... \, )   \\
                 &+z_{2k-i}        (&x_i z_{i+3}      &+y_iz_{i+2}z_4 +\, ... \, )   \\
                 &+z_{2k-(i+1)} (&x_{i+1}z_{i+4} &+y_{i+1}z_{i+3}z_4 +\, ... \, )   \\
&...&& \\
                 &+z_{k+3} (&x_{k-3}z_k &+y_{k-3}z_{k-1}z_4 +\, ... \, )   \\
                 &+z_{k+2} (&x_{k-2}z_{k+1} &+y_{k-2}z_k z_4 +\, ... \, )   \\
                 &&&+y_{k-1} z_{k+1}^2z_4  +\, ... \, .
\end{array}
\ee
Here $\{ x_i, y_i\}$ are constant coefficients.

First, by considering the $z_{2k}$ term in $\mce(F)$, we get
\be\label{eq:x10}
 x_1+(a_{2k,0}+\frac{3}{2} -\frac{2k}{3}a_{3,0} )x_0=0.
\ee
 
In order to extract the compatibility equations imposed on the $x_i$-coefficients only, from now on we compute modulo the curvature $R=\gamma^2-2 h_3\hb_3$. 
It means that we identify $$h_3\hb_3\equiv\frac{\gamma^2}{2}.$$

Consider the Jacobi equation $$\mce(F)\mod R,$$ and 
collect the equations from the coefficients of the monomials $z_{2k}, z_{2k-1}z_4,  z_{2k-2}z_5, \, ... $
up to $z_{k+2} z_{k+1} .$
i.e. the terms with the coefficient $x_j$'s (let's call them the principal terms). 
It is easily checked that when acted on by the Jacobi operator,
the terms not appearing in the above expansion do not have any contribution
to the principal terms.
Also, since $\del_{\omega}z_4\equiv 0\mod R$, by computing mod $R$ we eliminate the contributions from the $y_j$ coefficients, as claimed above. It follows that one may evaluate 
$$\mce(\{\tn{principal terms}\})\mod R,$$
and check only the principal terms in the image. 
This would yield a set of linear equations among the coefficients $x_j$'s.

Recall the following  formulas:
\begin{align*}
T_{j+1}&=   \sum_{s=0}^{j-3}
a_{j,s} \,  h_{j-s} \, \partial^{s}_{\xi} R,    \; \; \; \mbox{for} \; \; j\geq 3,  \n   \\
& \quad  \;\;a_{j,s} =\frac{(j+2s+3) }{2(j-1)} {j-1 \choose s+2}, \n\\
& \quad  \partial^{s}_{\xi}  R=\delta_{0s}\gamma^2 - 2 h_{3+s} \hb_3, \n\\
\ed z_j&\equiv \left(z_{j+1}-\frac{j}{3} z_4 z_j \right) \omega
+ \hat{T}_{j} \tn{r}^{-\frac{2}{3}}\omb \;\mod \iinfh, \quad{\rm for} \; j \geq 4.\n
\end{align*}
Note that, for $j\geq 4$,
\begin{align}
\delxb\delx(z_j)&\equiv\hat{T}_{j+1}-\frac{j}{3}(  z_4 \hat{T}_j)  \mod R, \n\\
\hat{T}_{j+1}&\equiv-\gamma^2 a_{j,j-3} z_j  +\tn{ (quadratic terms in $z_i$'s) }\mod R.\n
\end{align}
Since the principal terms except $z_{2k}$ are quadratic in the balanced coordinates $z_j$'s, 
the term $-\gamma^2 a_{j,j-3} z_j $  is the only form of contribution to the principal terms 
 from $\delxb z_{j+1}$ when $j+1<2k$.

\two
With this preparation, a direct computation yields the  following formulas for the principal terms.
We record only the relevant terms (here we set the scaling factor
 $\gamma^2=1$ temporarily for simplicity).
\begin{align*}
-\mce(z_{2k})&\equiv  ( a_{2k,2k-3}  -\frac{3}{2}) \,  z_{2k}
                        + ( a_{2k,2k-4} +  a_{2k,1}   - \frac{2k}{3}a_{2k-1,2k-4}) \,z_4 z_{2k-1}       \\
&\qquad+\sum_{j=2}^{k-2}   ( a_{2k,2k-j-3} +a_{2k,j}) \, z_{j+3} z_{2k-j} ,\\  
-\mce(z_4z_{2k-1})&\equiv
\begin{cases}
&\quad\cdot \\
&(a_{4,1}+a_{2k-1,2k-4}-\frac{3}{2}) \,z_{4}z_{2k-1}\\
&(a_{2k-2,2k-5}) \,z_{5}z_{2k-2}
\end{cases},  \\          
&... \\
-\mce(z_{j+3}z_{2k-j})&\equiv
\begin{cases}
&( a_{j+2,j-1} ) \,z_{j+2}z_{2k-(j-1)}  \\
&(a_{j+3,j}+a_{2k-j,2k-(j+3)}-\frac{3}{2}) \,z_{j+3}z_{2k-j}\\
&(a_{2k-(j+1),2k-(j+4)}) \,z_{j+4}z_{2k-(j+1)}
\end{cases}, \\
&... \\
-\mce(z_{k+1}z_{k+2})&\equiv 
\begin{cases}
& (a_{k,k-3}) \,z_{k}z_{k+3}  \\
& (2 a_{k+1,k-2}+a_{k+2,k-1}-\frac{3}{2}) \,z_{k+1}z_{k+2} \\
&\quad\cdot \\
\end{cases}, \mod R.  
\end{align*}
Note $a_{j+3,j}=\frac{3}{2}$ for all $j\geq0.$
Hence, except for the terms from $\mce(z_{2k})$ and the last term $z_{k+1}z_{k+2}$,
all the terms have the equal coefficient $\frac{3}{2}$.

The resulting set of linear equations on the coefficients $x_j$'s  
are the following system of three term relations, including the initial equation  \eqref{eq:x10}.
\be\begin{array}{rrrl}
\cdot&\cdot& x_1 
&+ (a_{2k,2k-3}+a_{2k,0} -k )x_0=0, \\
 \cdot &x_1&+x_2
&+  (a_{2k,2k-4} +  a_{2k,1}   - k )x_0=0, \\
x_1 &+x_2&+x_3
&+  (a_{2k,2k-5} +  a_{2k,2})x_0=0, \\
&&&...\\
 x_{j-1}&+x_j&+x_{j+1}&+ (a_{2k,2k-j-3} +a_{2k,j})x_0=0,\\
&&&...
\\
 x_{k-4}&+x_{k-3}&+x_{k-2}
&+ (a_{2k,k} +a_{2k,k-3})x_0=0,\\
 x_{k-3}&+x_{k-2}&+x_{k-2}
&+  (a_{2k,k-1} +a_{2k,k-2})x_0=0.
\end{array}\ee

\one
It is left to show  that this system of $k-1$ linear equations on the set of $k-1$ coefficient 
$\{ x_0, x_1, \, ... \, x_{k-2}\}$ has full rank. 
Set $t_j$ be the coefficient of $x_0$ of the $j$-th equation above, i.e.,
\begin{align*}
t_0&=a_{2k,2k-3}+a_{2k,0} -k , \n \\
t_1&=a_{2k,2k-4} +  a_{2k,1}   - k, \n \\
t_j&=a_{2k,2k-j-3} +a_{2k,j}, \quad \tn{for}  \;j\geq 2.\n
\end{align*}
A direct computation shows that the determinant $\chi_k$ of the $(k-1)$-by-$(k-1)$ matrix for the set of linear equations is given by
$$\pm \chi_k =\sum_{j=0}^{k-2}  \epsilon^j_k t_j,
$$
where
$$\epsilon^j_k=
\begin{cases}
&-2\;\;\tn{when}\; j\equiv k \mod 3 \\
&+1\;\;\tn{otherwise}.
\end{cases}$$
A  computation mod 3 shows that
\be\pm \chi_k =
\begin{cases}
& \frac{1}{2}\;\;\tn{when}\; k\equiv 0 \mod 3 \\
&1\;\;\tn{otherwise}.
\end{cases}
\ee

\subsubsection{Proof of Thm.\ref{thm:classifyJacobi}}\label{sec:proof}
$\,$

\quad b)\;
By Prop.\ref{prop:splitting} and the analysis in \S\ref{sec:interpseudo},
the higher-order pseudo-Jacobi fields correspond to the `Jacobi fields'
for the elliptic Tzitzeica equation studied in \cite{Fox2011,Fox2012}.
It follows from the classification result in \cite[Theorem 8.1]{Fox2012}.

\quad a)\;
From the analysis above, there exist no even order Jacobi fields.

We remark, without giving the proofs, that
the higher-order analogues of the results in \S\S\ref{sec:classicalcvlaw}, \ref{sec:cdcvlaw} are true:
the space of (differentiated) conservation laws injects into the space of Jacobi fields
under the natural \emph{symbol map} given by 
the differential of the associated spectral sequence; 
see \S\ref{sec:highercvlaws1}.
Thus, Prop.\ref{prop:splitting} also implies 
the corresponding splitting theorem for the higher-order conservation laws.

In particular,  
a higher-order conservation law of the minimal Lagrangian system
corresponds  to a  higher-order conservation law  of 
the  elliptic Tzitzeica equation.
From this, the induction argument using 
the recursion operators $\mcp, \mcn$ in \cite{Fox2012} 
shows that
there are no Jacobi fields (for the  minimal Lagrangian system) 
of degree $0\mod 6$ (or  of order $3\mod 6$).\ftmark\fttext{The obstruction 
to the application of the recursion operators $\mcp, \mcn$ lies
in the space of higher-order conservation laws of even weight.
This vanishes by the results from \cite{Fox2011,Fox2012}.}

The sequence of Jacobi fields of the given degree $d\equiv 2, 4\mod 6$ will be constructed in \S\ref{sec:formalKilling}.
\hfill$\square$

\section{Formal Killing fields}\label{sec:formalKilling}
Recall that the original differential system for minimal Lagrangian surfaces
is defined on the bundle of Lagrangian planes $X\to M$.
It is a 6-symmetric space associated with the Lie group $\SL(3,\C)$,
and the minimal Lagrangian surfaces arise as the primitive harmonic maps.
From the theory of integrable systems, this implies that
the $\g$-valued Maurer-Cartan form $\psi$, \eqref{eq:psiform},
admits an extension to the  
$\g^{\C}[\lambda^{-1},\lambda]$-valued ($\g^{\C}=\sla(3,\C)$) 
extended Maurer-Cartan form $\psi_{\lambda}$, \eqref{eq:primitiveform} in the below, 
by inserting the spectral parameter $\lambda.$
The structure equation for the minimal Lagrangian system shows that
$\psi_{\lambda}$ is compatible and satisfies the Maurer-Cartan equation.

In this section, 
we give a construction of the corresponding $\g^{\C}[[\lambda]]$-valued 
canonical formal Killing fields  
to utilize this aspect of symmetry of the minimal Lagrangian system.
The construction relies on the  pair of 3-step recursions between Jacobi fields and pseudo-Jacobi fields
which are embedded in the structure equation for the formal Killing fields.
We give the differential algebraic   inductive  formulas
for the pair of formal Killing fields
that correspond to two particular sets of  initial data.
As a consequence, 
we will be able to read off the infinite sequence of higher-order (pseudo) Jacobi fields 
and conservation laws
from the components of the formal Killing fields.
 
In hindsight, the recursion relations were anticipated from the structure equation 
for the classical Killing field \eqref{eq:classicalKF}.
Note that the 6-step recursion introduced in \cite{Fox2012}
is the union of these two 3-step recursions  
when translated to our setting.

\two
In \S\ref{sec:KFstrt}, 
we record the structure equation for the formal Killing fields
with respect to the extended Maurer-Cartan form.
In \S\ref{sec:KFinitial}, 
we determine the first few terms of the formal Killing fields 
for the two initial ans\"atze given by Exam.\ref{exam:initialdata}.
The relevant observation is that the coefficient $h_3$ of Hopf differential 
provides the lower-end terms for the formal Killing fields,
which allow one to truncate the  terms of negative $\lambda$-degrees.
With this preparation, 
we give in \S\ref{sec:KFformulae} the inductive formula
for the formal Killing field for each set of the initial data.

\subsection{Structure equation}\label{sec:KFstrt}
In this section, we introduce the extended Maurer-Cartan form $\psi_{\lambda}$
and record the recursive structure equation (mod $\iinfh$) 
for the coefficients of the $\g^{\C}[[\lambda]]$-valued formal Killing field
with respect to $\psi_{\lambda}$.
\subb{Extended Maurer-Cartan form}
Consider the $\g$-valued 1-form $\psi$, Eq.\eqref{eq:psiform}.
Evaluating mod $\iinfh$ (i.e., $\eta_2=h_3\xi, \etab_2=\hb_3\xib$), it becomes
\[\psi=\psi_{+}+\psi_{0}+\psi_{-},
\]
where
\be
\psi_{-} =\frac{1}{2} \left[ \begin {array}{ccc} \cdot &- \gamma
&\im \gamma  \n \\
\noalign{\medskip}  \gamma &\im h_{{3}} & -  h_{{3}}  \\
\noalign{\medskip} -\im \gamma &- h_{{3}} &- \im h_{{3}} \end {array} \right]\xi,
\ee
\be \psi_{0} =
\left[ \begin {array}{ccc} \cdot &\cdot & \cdot  \n \\
\cdot &\cdot& \rho   \\
\cdot  & -\rho &\cdot
\end {array} \right], \n
\ee
\be \psi_{+} =\frac{1}{2}  \left[ \begin {array}{ccc} 0&- \gamma &-\im \gamma \\
\noalign{\medskip}  \gamma & \im \hb_3 & \hb_3 \\
\noalign{\medskip}\im \gamma &\hb_3 &-\im \hb_3
\end {array} \right]\xib. \n
\ee

\one 
Let $\lambda\in\C^*$ be the auxiliary spectral parameter. 
\begin{defn}
The \tb{extended Maurer-Cartan form} is the 
$\g^{\C}[\lambda^{-1},\lambda]$-valued  1-form on $\Fh{\infty}$ given by
\be\label{eq:primitiveform}
\psi_{\lambda}:=\lambda \psi_{+}+\psi_0+\lambda^{-1} \psi_-.
\ee
\end{defn}
The extended 1-form $\psi_{\lambda}$ takes values in the Lie algebra $\g$ 
when $\lambda$ is a unit complex number. 
It  satisfies the structure equation
\be\label{eq:primitiveMC}
\ed\psi_{\lambda}+\psi_{\lambda}\w\psi_{\lambda}\equiv 0\mod\iinfh.
\ee
 
\subb{Formal Killing field}
Recall the decomposition of the Lie algebra $\g^{\C}$ given in Fig.\ref{fig:Xdecompo0}.
By expanding each of the scalar coefficients
 $\{ \, \tb{p}, \tb{b}, \tb{c}, \tb{f}, \tb{a}, \tb{g}, \tb{s}, \tb{t}   \,\}$ 
as a series in $\C[[\lambda]]$,
we give an abridged definition of the formal Killing fields 
associated with the extended Maurer-Cartan form.
\begin{defn}\label{defn:FK}
Let $\psi_{\lambda}$, \eqref{eq:primitiveform}, be 
the extended Maurer-Cartan form.
A \tb{formal Killing field}   is a function\ftmark\fttext{We will find that the canonical formal Killing fields to be constructed  are defined on the open subset  
$\Fh{\infty}_{**}=\{ h_3\ne 0, \infty, \; h_j\ne\infty\;\forall j\geq 4\}\subset\Fh{\infty}_{*}.$}
\be 
\tb{X}_{\lambda}: \Fh{\infty} \to\g^{\C}[[ \lambda]],\n
\ee
such that;
\begin{enumerate}[\qquad a)]
\item
it satisfies the Killing field equation
\be\label{eq:KillingEquation}
\ed \tb{X}_{\lambda}+[\psi_{\lambda}, \tb{X}_{\lambda}]\equiv 0\mod\iinfh,
\ee
\item
its components are the formal series in $ \lambda$ given explicitly by
\be\label{eq:KFcomponents}
\begin{array}{rlrl}
\tb{p}&=\sum p^{6k+4}\lambda^{6k+2},
& \tb{a}&= \sum a^{6k+7}\lambda^{6k+5},        \\
\tb{b}&=\sum b^{6k+5}\lambda^{6k+3},
& \tb{g}&= \sum g^{6k+2}\lambda^{6k},   \\
\tb{c}&=\sum c^{6k+5}\lambda^{6k+3},
&\tb{s}&=\sum  s^{6k+3}\lambda^{6k+1},   \\
\tb{f}&=\sum f^{6k+6}\lambda^{6k+4},
&\tb{t}&= \sum t^{6k+3}\lambda^{6k+1}.
\end{array}
\ee
Here the sums are over the integer index $k$  from $0$ to $\infty$.
\end{enumerate}
\end{defn}

\subb{Recursive structure equation}
When the Killing field equation \eqref{eq:KillingEquation} is expanded as a series in $\lambda$,
it implies the following recursive structure equation.

\one
[\tb{$n$-th equation}]
\begin{align}\label{eq:formalKilling_n}
 \ed  p^{6n+4} &=(\im\gamma b^{6n+5}+2\im h_3 c^{6n+5}  )\xi
                          + ( \im\gamma s^{6n+3}+2\im \hb_3 t^{6n+3})\xib,     \\
 \ed  b^{6n+5} +\im b^{6n+5}\rho
                         &=\im h_3 f^{6n+6} \xi+ \frac{\im}{2}\gamma p^{6n+4} \xib,    \n\\
 \ed  c^{6n+5} -2 \im c ^{6n+5}\rho&= \im \gamma f^{6n+6}\xi  +\im \hb_3 p^{6n+4}  \xib,    \n\\
 \ed  f^{6n+6} -\im f^{6n+6}\rho  &=\frac{3\im}{2}\gamma a^{6n+7} \xi
                                                          +(\im \gamma  c^{6n+5}+\im\hb_3b^{6n+5})\xib, \n\\
 \ed  a^{6n+7}  &=\im \gamma g^{6n+8}\xi + \im \gamma f^{6n+6}\xib,  \n\\
 \ed  g^{6n+8}+\im g^{6n+8}\rho &= (-\im \gamma t^{6n+9}-\im h_3 s^{6n+9})\xi
                                                             +\frac{3\im}{2}\gamma a^{6n+7}\xib,   \n\\
 \ed  s^{6n+9} -\im s^{6n+9}\rho&= \frac{\im}{2} \gamma p^{6n+10}\xi -\im\hb_3 g^{6n+8}\xib,  \n\\
 \ed  t^{6n+9}+2 \im t^{6n+9}\rho&=\im h_3 p^{6n+10}\xi-\im \gamma g^{6n+8}\xib,
\qquad\qquad  \quad (\tn{mod}\;\iinfh).  \n
\end{align}

\two
The relevance of this formal structure equation for the analysis of the minimal Lagrangian system
lies in the following observation.
\begin{lem}\label{lem:formalJacobi}
Suppose the coefficients
$\{p^{6n+4}, b^{6n+5},\,c^{6n+5}, f^{6n+6}, a^{6n+7},
 g^{6n+8}, s^{6n+9}, t^{6n+9} \}$
satisfy the recursive  structure equation \eqref{eq:formalKilling_n}. Then,
\begin{enumerate}[\qquad a)]
\item  $p^{6n+4}$ is a pseudo-Jacobi field.
\item  $a^{6n+7}$ is a Jacobi field.
\end{enumerate}
\end{lem}
The lemma   indicates that
one may obtain a canonical sequence of (pseudo) Jacobi fields 
by \emph{solving} the structure equation \eqref{eq:formalKilling_n}.

\subsection{Initial analysis}\label{sec:KFinitial}
Recall from Exam.\ref{exam:initialdata} that 
$z_4$ is a pseudo-Jacobi field, and $z_5-\frac{5}{3}z_4^2$ is a Jacobi field.
We start the process of solving for the canonical formal Killing fields
by determining the first few terms generated by these initial data.
By Lem.\ref{lem:delbpoly0}, the coefficients of the resulting formal Killing fields  are 
the elements in the polynomial ring $\C[z_4, z_5, \, ... \, ]$, 
up to scaling by the appropriate powers of $h_3^{\frac{1}{3}}$. 

It turns out that 
the pair of formal Killing fields generated by $z_4$, and $z_5-\frac{5}{3}z_4^2$
are sufficient to cover all of the infinite sequence of higher-order (pseudo) Jacobi fields.

\subsubsection{Case $p^4=z_4$}\label{sec:521}
Set $g^2=0.$ By inspection, set
$$s^3=-\frac{3\im}{2}\gamma h_3^{-\frac{1}{3}},  \quad t^3= \frac{3\im}{2}h_3^{\frac{2}{3}}.$$
Differentiating this, we get $$p^4=z_4$$ as expected.

Solving the equation $\delxb b^5= \frac{\im}{2}\gamma z_4$, we get
$$b^5= -\frac{\im}{3\gamma} h_3^{\frac{1}{3}} (z_5 -\frac{5}{3}z_4^2).
$$
From the equation $\delx p^4= \im\gamma b^{5}+2\im h_3 c^{5}$, this implies
$$c^5=-\frac{\im}{3}h_3^{-\frac{2}{3}}(z_5 -\frac{7}{6}z_4^2).
$$

Successive derivatives of $b^5$ give 
\begin{align}
f^6&= -\frac{1}{3\gamma} h_3^{-\frac{1}{3}} (z_6-\frac{14}{3}z_5 z_4+\frac{35}{9}z_4^3),           \n\\
a^7&= \frac{2\im}{9\gamma^2}
(z_7-7 z_6 z_4  -\frac{14}{3} z_5^2+\frac{245}{9}z_5 z_4^2  - \frac{455}{27}z_4^4),       \n\\
g^8&= \frac{2}{9\gamma^3} h_3^{\frac{1}{3}} ( z_{{8}}-{\frac {28}{3}}\,z_{{7}}z_{{4}}-{\frac {49}{3}}\,z_{{6}}z_{{5}}
+{\frac {455}{9}}\,z_{{6}} z_{{4}}^{2}
+70 z_{{5}}^{2}z_{{4}}-{\frac {5005}{27}}z_{{5}} z_{{4}}^{3} +{\frac {7280}{81}}\, z_{{4}}^{5}). \n
\end{align}

Proceed with the similar computation  as above by solving (by inspection) 
the associated $\delxb$-equation,
and we get
\begin{align}
s^9&=\frac{2\im}{27\gamma^3}h_3^{-\frac{1}{3}}\Big(
z_{{9}}  -11z_{{8}}z_{{4}}-{\frac {79}{3}}z_{{7}}z_{{5}} +{\frac {689}{9}}z_{{7}}z_{{4}}^{2}
-16z_{{6}}^{2} +286z_{{6}}z_{{5}}z_{{4}} -{\frac{3380}{9}}z_{{6}}z_{{4}}^{3}      \n\\
&\qquad\qquad\;\; +{\frac {1976}{27}}z_{{5}}^{3}
 -{\frac {22360}{27}}z_{{5}}^{2}z_{{4}}^{2}+{\frac {108680}{81}}z_{{5}}z_{{4}}^{4}
-{\frac {380380}{729}}z_{{4}}^{6}
\Big), \n \\
t^9&=\frac{4\im}{27\gamma^4}h_3^{\frac{2}{3}}
\Big(z_{{9}}-12 z_{{8}}z_{{4}}-{\frac {76}{3}}z_{{7}}z_{{5}}+{\frac {758}{9}}z_{{7}}z_{{4}}^{2}
-{\frac {33}{2}}z_{{6}}^{2}+{\frac {901}{3}}z_{{6}}z_{{5}}z_{{4}}
-{\frac {3770}{9}}z_{{6}}z_{{4}}^{3}\n\\
&
\qquad\qquad\;\; +{\frac {1847}{27}}z_{{5}}^{3}-{\frac {47255}{54}} z_{{5}}^{2}z_{{4}}^{2}
+{\frac {120380}{81}}z_{{5}}z_{{4}}^{4}-{\frac {432250}{729}}z_{{4}}^{6}
\Big),             \n   \\
p^{10}&=\frac{4}{27\gamma^4}\Big(
z_{{10}} -{\frac {43}{3}}z_{{9}}z_{{4}} -{\frac {112}{3}}z_{{8}}z_{{5}}+{\frac {1118}{9}}z_8 z_{{4}}^{2}
-{\frac {175}{3}}z_{{7}}z_{{6}}+{\frac {4979}{9}}z_{{7}}z_{{5}}z_{{4}}-{\frac {21164}{27}}z_{{7}}z_{{4}}^{3} \n\\
&
\qquad\quad\;\;
+{\frac {1066}{3}}z_{{6}}^{2}z_{{4}} +{\frac {4550}{9}}z_{{6}}z_{{5}}^{2}
-{\frac {116324}{27}}z_{{6}}z_{{5}}z_{{4}}^{2} +{\frac {301340}{81}}z_{{6}}z_{{4}}^{4}  \n\\
&
\qquad\quad\;\; -{\frac {165776}{81}}z_{{5}}^{3}z_{{4}}
+{\frac {286520}{27}}z_{{5}}^{2}z_{{4}}^{3}
-{\frac {3151720}{243}}z_{{5}}z_{{4}}^{5}
+{\frac {9509500}{2187}}z_{{4}}^{7} \Big).
\n
\end{align}
By Lem.\ref{lem:formalJacobi}, 
$a^7$ is a Jacobi field, and $p^{10}$ is a pseudo-Jacobi field.

\subsubsection{Case $a^5=z_5-\frac{5}{3}z_4^2$}\label{sec:522}
For the formal Killing field generated by the Jacobi field $z_5-\frac{5}{3}z_4^2$, it is convenient
to lower the upper indices of the formal Killing field coefficients by 2 
to  match the  order.
The resulting structure equation is recorded as follows.
\be\label{eq:KFcomponents'}
\begin{array}{rlrl}
\tb{p}&=\sum p^{6k+2}\lambda^{6k}, & \tb{a}&= \sum a^{6k+5}\lambda^{6k+3},        \\
\tb{b}&=\sum b^{6k+3}\lambda^{6k+1}, & \tb{g}&= \sum g^{6k+6}\lambda^{6k+4},   \\
\tb{c}&=\sum c^{6k+3}\lambda^{6k+1}, &\tb{s}&=\sum  s^{6k+7}\lambda^{6k+5},   \\
\tb{f}&=\sum f^{6k+4}\lambda^{6k+2},   &\tb{t}&= \sum t^{6k+7}\lambda^{6k+5}.
\end{array}
\ee

[\tb{$n$-th equation}']
\begin{align}\label{eq:formalKilling_n'}
 \ed  p^{6n+2} &=(\im\gamma b^{6n+3}+2\im h_3 c^{6n+3}  )\xi
                          + ( \im\gamma s^{6n+1}+2\im \hb_3 t^{6n+1})\xib,     \\
 \ed  b^{6n+3} +\im b^{6n+3}\rho
                         &=\im h_3 f^{6n+4} \xi+ \frac{\im}{2}\gamma p^{6n+2} \xib,    \n\\
 \ed  c^{6n+3} -2 \im c ^{6n+3}\rho&= \im \gamma f^{6n+4}\xi  +\im \hb_3 p^{6n+2}  \xib,    \n\\
 \ed  f^{6n+4} -\im f^{6n+4}\rho  &=\frac{3\im}{2}\gamma a^{6n+5} \xi
                                                          +(\im \gamma  c^{6n+3}+\im\hb_3b^{6n+3})\xib, \n\\
 \ed  a^{6n+5}  &=\im \gamma g^{6n+6}\xi + \im \gamma f^{6n+4}\xib,  \n\\
 \ed  g^{6n+6}+\im g^{6n+6}\rho &= (-\im \gamma t^{6n+7}-\im h_3 s^{6n+7})\xi
                                                             +\frac{3\im}{2}\gamma a^{6n+5}\xib,   \n\\
 \ed  s^{6n+7} -\im s^{6n+7}\rho&= \frac{\im}{2} \gamma p^{6n+8}\xi -\im\hb_3 g^{6n+6}\xib,  \n\\
 \ed  t^{6n+7}+2 \im t^{6n+7}\rho&=\im h_3 p^{6n+8}\xi-\im \gamma g^{6n+6}\xib,
\qquad\qquad  \quad (\tn{mod}\;\iinfh).  \n
\end{align}

We proceed to solve for the first few terms.

Let $p^2=0.$ By inspection, set
$$b^3=-\frac{9}{2}\gamma h_3^{\frac{1}{3}}, \quad c^3= \frac{9}{4}\gamma^2 h_3^{-\frac{2}{3}}.$$
Differentiating these equations successively, one gets
$$f^4=\frac{3\im}{2}\gamma h_3^{-\frac{1}{3}}z_4,\quad
a^5=z_5-\frac{5}{3}z_4^2,\quad
g^6 =-\frac{\im}{\gamma}   h_3^{\frac{1}{3}}
\left( z_6  -5 z_5z_4+\frac{40}{9} z_4^3 \right).$$
Note that $a^5$ is as expected.

A similar computation as in the previous case yields,
\begin{align}
s^7&=\frac{1}{3\gamma}h_3^{-\frac{1}{3}}
\Big( z_7 -6z_6z_4-\frac{16}{3} z_5^2+\frac{220}{9} z_5z_4^2-\frac{385}{27} z_4^4 \Big), \n \\
t^7&=\frac{2}{3\gamma^2}h_3^{\frac{2}{3}}\Big( z_{{7}}-7z_{{6}}z_{{4}}-{\frac {29}{6}}z_{{5}}^{2}
          +{\frac {250}{9}}z_{{5}}z_{{4}}^{2}  -{\frac {935}{54}}z_{{4}}^{4} \Big), \n   \\ 
p^8&=-\frac{2\im}{3\gamma^2}\Big(z_{{8}}-{\frac {26}{3}}z_{{7}}z_{{4}}
 -{\frac {50}{3}}z_{{6}}z_{{5}}+{\frac {418}{9}}z_{{6}}z_{{4}}^{2}+{\frac {616}{9}}z_{{5}}^{2}z_{{4}}
 -{\frac {1540}{9}}z_{{5}}z_{{4}}^{3}+{\frac {6545}{81}}z_{{4}}^{5} \Big),\n\\
b^9&=-\frac{2}{9\gamma^3}h_3^{\frac{1}{3}} \Big(
z_{{9}}-12z_{{8}}z_{{4}} -{\frac {74}{3}}z_{{7}}z_{{5}}+{\frac {748}{9}}z_{{7}}z_{{4}}^{2}
-17z_{{6}}^{2}+{\frac {902}{3}}z_{{6}}z_{{5}}z_{{4}}
-{\frac {3740}{9}}z_{{6}}z_{{4}}^{3} \n\\
&\qquad\qquad\;\;\;
+{\frac {1760}{27}}z_{{5}}^{3}
-{\frac {23320}{27}}z_{{5}}^{2}z_{{4}}^{2}
+{\frac {118745}{81}}z_{{5}}z_{{4}}^{4}
-{\frac {425425}{729}}z_{{4}}^{6}\Big),   \n   \\ 
c^9&=-\frac{2}{9\gamma^2}h_3^{-\frac{2}{3}} \Big(
z_{{9}}-11z_{{8}}z_{{4}}-{\frac {77}{3}}z_{{7}}z_{{5}}+{\frac {682}{9}}z_{{7}}z_{{4}}^{2}
-{\frac {33}{2}}z_{{6}}^{2}      +286z_{{6}}z_{{5}}z_{{4}}
-374z_{{6}}z_{{4}}^{3}  \n \\
&\qquad\qquad\;
+{\frac {1892}{27}}z_{{5}}^{3}
-{\frac {22066}{27}} z_{{5}}^{2}z_{{4}}^{2}
+{\frac {107525}{81}}z_{{5}}z_{{4}}^{4}
-{\frac {752675}{1458}}z_{{4}}^{6}\Big).\n 
\end{align} 
Successively differentiating $b^9, c^9$, one finally gets,
\begin{align*}
f^{10}&=\frac{2\im}{9\gamma^3}h_3^{-\frac{1}{3}}\Big(
z_{10}-{\frac {44}{3}}z_{{9}}z_{{4}}-{\frac {110}{3}}z_{{8}}z_{{5}} +{\frac {1144}{9}}z_{{8}}z_{{4}}^{2}
-{\frac {176}{3}}z_{{7}}z_{{6}}
+{\frac {1672}{3}}z_{{7}}z_{{5}}z_{{4}}
-{\frac {21692}{27}}z_{{7}}z_{{4}}^{3} \n\\
&\qquad\qquad\;
+363z_{{6}}^{2}z_{{4}}
+{\frac {4466}{9}}z_{{6}}z_{{5}}^{2}
-{\frac {118184}{27}}z_{{6}}z_{{5}}z_{{4}}^{2}
+{\frac {309485}{81}}z_{{6}}z_{{4}}^{4}
-{\frac {164560}{81}}z_{{5}}^{3}z_{{4}} \n\\
&\qquad\qquad\;
+{\frac {871420}{81}}z_{{5}}^{2}z_{{4}}^{3}
-{\frac {1075250}{81}}z_{{5}}z_{{4}}^{5}
+{\frac {9784775}{2187}}z_{{4}}^{7}            \Big), \n\\ 
a^{11}&=\frac{4}{27\gamma^4} \Big(
z_{11}-{\frac {55}{3}}z_{{10}}z_{{4}}-{\frac {154}{3}}z_{{9}}z_{{5}}+{\frac {1760}{9}}z_{{9}}z_{{4}}^{2}
-{\frac {286}{3}}z_{{8}}z_{{6}}+{\frac {2948}{3}}z_{{8}}z_{{5}}z_4
-{\frac {41140}{27}}z_{{8}}z_{{4}}^{3}
\\
&\qquad\quad\;
-{\frac {176}{3}}z_{{7}}^{2}
+{\frac {14014}{9}}z_{{7}}z_{{6}}z_{{4}}
+{\frac {9482}{9}}z_{{7}}z_{{5}}^{2}
-{\frac {268532}{27}}z_{{7}}z_{{5}}z_{{4}}^{2}
+{\frac {247775}{27}}z_{{7}}z_{{4}}^{4} \\
&\qquad\quad\;
+{\frac {12199}{9}}z_{{6}}^{2}z_{{5}}
-{\frac {173723}{27}}z_{{6}}^{2}z_{{4}}^{2}
-{\frac {158950}{9}}z_{{6}}z_{{5}}^{2}z_{{4}}
+{\frac {5344460}{81}}z_{{6}}z_{{5}}z_{{4}}^{3}
-{\frac {10343905}{243}}z_{{6}}z_{{4}}^{5}
\\
&\qquad\quad\;
-{\frac {164560}{81}}z_{{5}}^{4}
+{\frac {11133980}{243}}z_{{5}}^{3}z_{{4}}^{2}
-{\frac {36171410}{243}}z_{{5}}^{2}z_{{4}}^{4}
+{\frac {320101925}{2187}}z_{{5}}z_{{4}}^{6}
-{\frac {283758475}{6561}}z_{{4}}^{8}
\Big).\n
\end{align*}
By Lem.\ref{lem:formalJacobi}, 
$a^5, a^{11}$ are Jacobi fields, and $p^{8}$ is a pseudo-Jacobi field.

\subsection{Inductive formulas}\label{sec:KFformulae}
Based on the initial analyses given above, 
we give the differential algebraic inductive formulas
for the respective formal Killing fields:

\one
\qquad \ref{sec:p4z4})\quad  Case $p^4=z_4$, 

\qquad \ref{sec:a5z5})\quad   Case $a^5=z_5-\frac{5}{3}z_4^2$.
 
\one
Note from Eqs.\eqref{eq:formalKilling_n}, \eqref{eq:formalKilling_n'} that 
one needs to solve for the coefficients $\{ s^*, t^*\}$, and $\{b^*, c^*\}$.

\subsubsection{Case $p^4=z_4$}\label{sec:p4z4}
Assume the initial data from \S\ref{sec:521}.

\two\noi
\tb{[Formulas for $\,s^{6n+3}, t^{6n+3}$]}.\,
Suppose all the coefficients up to $g^{6n+2}$ are known, $n\geq 1$.
We give a formula for $\{s^{6n+3}, t^{6n+3}\}.$

Set the truncated formal Killing field
$$ \tb{X}_{6n+2}:=
\left[ \begin {array}{ccc} -2\im \tb{a} &\tb{b}+\tb{f}+\tb{g}-\tb{s}
&\im \tb{b}- \im \tb{f}+\im \tb{g}+\im \tb{s}\\
\noalign{\medskip}-\tb{b}+\tb{f}+\tb{g}+\tb{s} & \im \tb{c}+\im \tb{a}-\im \tb{t}&-\tb{p}+\tb{c}+\tb{t}\\
\noalign{\medskip}-\im \tb{b}-\im \tb{f} +\im \tb{g}-\im \tb{s}
&\tb{p}+\tb{c}+\tb{t}&-\im \tb{c}+\im \tb{a}+\im \tb{t}\end {array} \right],
$$
where
\be\label{eq:components1}
\begin{array}{rlrl}
\tb{p}&=\sum_{k=0}^{n} p^{6k+4}\lambda^{6k+2},     
& \tb{a}&= \sum_{k=0}^{n-1} a^{6k+7}\lambda^{6k+5},         \\
\tb{b}&=\sum_{k=0}^{n-1} b^{6k+5}\lambda^{6k+3},  
& \tb{g}&= \sum_{k=0}^{n} g^{6k+2}\lambda^{6k},   \\
\tb{c}&=\sum_{k=0}^{n-1} c^{6k+5}\lambda^{6k+3},   
&\tb{s}&=\sum_{k=0}^{n} s^{6k+3}\lambda^{6k+1},   \\
\tb{f}&=\sum_{k=0}^{n-1} f^{6k+6}\lambda^{6k+4},    
&\tb{t}&= \sum_{k=0}^{n} t^{6k+3}\lambda^{6k+1}.
\end{array}
\ee
Here the unknown coefficients are $s^{6n+3}, t^{6n+3}, p^{6n+4}.$
The determinant   is given by
$$ \det(\tb{X}_{6n+2}) =\im( 4\tb{gsp} -4 \tb{fga}-4\tb{b$^2$c} -4\tb{f$^2$t}+4\tb{g$^2$c}
+4\tb{s$^2$t}+2\tb{a$^3$}-2\tb{ap$^2$}+8\tb{act}-4\tb{bsa}+4\tb{bfp}) .
$$
Expanding as a series in $\lambda$, let us denote
$$\det(\tb{X}_{6n+2}):=\sum_{j=0}^{3n} \tb{x}_{6n+2}^{6j+3} \lambda^{6j+3}.
$$

Consider now the derivative 
$$\delxb(\det(\tb{X}_{6n+2})).$$
The structure equation shows that this term stems from the absence of 
$b^{6n+5}, c^{6n+5}$-terms in $\tb{X}_{6n+2}$.
Hence only the terms that contain $p^{6n+4}$ contribute to
$\delxb(\det(\tb{X}_{6n+2}))$.

From the  determinant formula above, one finds by checking the $\lambda$-degree that
$$\delxb \tb{x}_{6n+2}^{6j+3}=0, \quad \tn{for}\; j\leq n.
$$
By Cor.\ref{cor:lemma5.4} and weighted homogeneity, this implies that
$$  \tb{x}_{6n+2}^{6j+3}=0, \quad \tn{for}\; j\leq n.
$$
Consider the term $\tb{x}_{6n+2}^{6n+3}$ of the highest $\lambda$-degree among these.  
We have
\begin{align} \label{eq:x6n+43}
\tb{x}_{6n+2}^{6n+3} &= 4\im(2s^3 t^3 s^{6n+3}+(s^3)^2t^{6n+3})+\tb{y}_{6n+2}^{6n+3},   
\end{align}
where $\tb{y}_{6n+2}^{6n+3} \in\mco(6n+2)$.

On the other hand, we have
\be\label{eq:delxg6n+2}
\delx g^{6n+2}=-\im h_3 s^{6n+3}-\im\gamma t^{6n+3}.
\ee
Combining \eqref{eq:x6n+43}, \eqref{eq:delxg6n+2}, one gets
\begin{align}\label{eq:st6n+3}
s^{6n+3}&= \frac{\im}{27\gamma}h_3^{-1}\left(9\gamma \delx g^{6n+2}
+ h_3^{\frac{2}{3}} \tb{y}_{6n+2}^{6n+3} \right),            \\
t^{6n+3}&= \frac{\im}{27\gamma^2}\left(18\gamma  \delx g^{6n+2}
- h_3^{ \frac{2}{3}} \tb{y}_{6n+2}^{6n+3} \right).                \n
\end{align}

\two\noi
\tb{[Formulas for $b^{6n-1}, c^{6n-1}$]}.\,
Suppose all the coefficients up to $p^{6n-2}$ are known, $n\geq 1$.
We give a formula for $\{b^{6n-1}, c^{6n-1}\}.$

\one
Given the truncated formal Killing field $\tb{X}_{6n+2}$ as above, let
$$ \det(\mu \tn{I}_3+\tb{X}_{6n+2})=\mu^3+\sigma_2(\tb{X}_{6n+2})\mu+\det(\tb{X}_{6n+2})
$$
be the characteristic polynomial. For the case at hand, we utilize $\sigma_2(\tb{X}_{6n+2}).$
It is given by the formula
$$\sigma_2(\tb{X}_{6n+2})
= 3\tb{a}^2+\tb{p}^2-4\tb{c}\tb{t}-4\tb{b}\tb{s}-4\tb{f}\tb{g}.
$$
Expanding  as a series in $\lambda$, let us denote
$$\sigma_2(\tb{X}_{6n+2}):=\sum_{j=0}^{2n} \tb{x}_{6n+2}^{6j+4} \lambda^{6j+4}.
$$

Consider the  derivative $\delxb( \sigma_2(\tb{X}_{6n+2}) )$.
By the similar argument as above, one finds that
$$\delxb \tb{x}_{6n+2}^{6j-2}=0, \quad \tn{for}\; j\leq n,
$$
and hence  
$$  \tb{x}_{6n+2}^{6j-2}=0, \quad \tn{for}\; j\leq n.
$$

\one
Consider the term $\tb{x}_{6n+2}^{6n-2}$.  Then
\begin{align} \label{eq:x6n-6}
\tb{x}_{6n+2}^{6n-2}&= -4 s^3 b^{6n-1} -4 t^3c^{6n-1} +\tb{y}_{6n+2}^{6n-2}.
\end{align}
Here $\tb{y}_{6n+2}^{6n-2} \in\mco(6n-2)$.

On the other hand, we have
\be\label{eq:delxp6n+4}
\delx p^{6n-2}= \im \gamma b^{6n-1}+2\im h_3c^{6n-1}.
\ee
Combining \eqref{eq:x6n-6}, \eqref{eq:delxp6n+4}, one gets
\begin{align}\label{eq:bc6n-1}
b^{6n-1}&= \frac{\im}{9\gamma}\left(-3\delx p^{6n-2}
+ h_3^{\frac{1}{3}} \tb{y}^{6n-2}_{6n+2} \right),            \\
c^{6n-1}&= -\frac{\im}{18}h_3^{-1}\left( 6\delx p^{6n-2}
+ h_3^{ \frac{1}{3}} \tb{y}^{6n-2}_{6n+2} \right).                \n
\end{align}

\begin{thm}\label{thm:FKformulaep4}
Given the ansatz $p^4=z_4$ and the initial data described in \S\ref{sec:521};
\begin{enumerate}[\qquad a)]
\item
there exists a $\g^{\C}[[\lambda]]$-valued canonical formal Killing field $\tb{X}(p^4)$ 
which extends these data.
The coefficients of its components are generated by the structure equation \eqref{eq:formalKilling_n},
and the differential algebraic  inductive formulas  \eqref{eq:st6n+3},\eqref{eq:bc6n-1}.
Equivalently, $\tb{X}(p^4)$ is determined by the  constraint,
$$\det(\mu \tn{I}_3+\tb{X}(p^4))=\mu^3+\left( \frac{27}{2}\gamma^2\right)\lambda^3.$$
Here $\tn{I}_3$ denotes the 3-by-3 identity matrix.
\item
each coefficient  of $\tb{X}(p^4)$ is an element in the polynomial ring $\C[z_4, z_5, \, ... \, ]$ 
up to scaling by appropriate powers of $h_3^{\frac{1}{3}}$.
\end{enumerate}
\end{thm}
\begin{cor}\label{cor:FKformulaep4}
Given the formal Killing field $\tb{X}(p^4)$,
the sequence of coefficients $$\{\, p^{6n+4}, \ol{p}^{6n+4} \,\}_{n=0}^{\infty}$$
are distinct higher-order pseudo-Jacobi fields,
and the sequence of coefficients $$\{\, a^{6n+7}, \ol{a}^{6n+7} \,\}_{n=0}^{\infty}$$
are distinct higher-order Jacobi fields.
\end{cor}

\subsubsection{Case $a^5=z_5-\frac{5}{3}z_4^2$}\label{sec:a5z5}
Recall that we follow \eqref{eq:KFcomponents'}, and  \eqref{eq:formalKilling_n'}.

Assume the initial data from \S\ref{sec:522}.
By the same analysis as in \S\ref{sec:p4z4},
we obtain the corresponding formal Killing field $\tb{X}(a^5)$.
\begin{thm}\label{thm:FKformulaea5}  
Given the ansatz $a^5=z_5-\frac{5}{3}z_4^2$ and the initial data described in \S\ref{sec:522};
\begin{enumerate}[\qquad a)]
\item
there exists a $\g^{\C}[[\lambda]]$-valued canonical formal Killing field $\tb{X}(a^5)$
which extends these data.
The coefficients of its components are determined by the structure equation \eqref{eq:formalKilling_n'},
and   the  constraint,
$$\det(\mu \tn{I}_3+\tb{X}(a^5))=\mu^3-\left(\frac{729}{4}\im\gamma^4 \right)\lambda^3.$$
\item
each coefficient  of $\tb{X}(a^5)$ is an element in the polynomial ring $\C[z_4, z_5, \, ... \, ]$ 
up to scaling by appropriate powers of $h_3^{\frac{1}{3}}$.
\end{enumerate}
\end{thm}
\begin{cor}\label{cor:FKformulaea5}
Given the formal Killing field $\tb{X}(a^5)$,
the sequence of coefficients $$\{\, a^{6n+5}, \ol{a}^{6n+5} \,\}_{n=0}^{\infty}$$
are distinct higher-order Jacobi fields,
and the sequence of coefficients $$\{\, p^{6n+8}, \ol{p}^{6n+8} \,\}_{n=0}^{\infty}$$
are distinct higher-order pseudo-Jacobi fields.
\end{cor}
 
\section{Higher-order conservation laws}\label{sec:highercvlaws1}
Recall that the classical conservation laws are defined  
as the elements in the 1-st characteristic cohomology of the quotient complex 
$$(\Omega^*(X)/\mci, \underline{\ed}).$$
Generalizing this, the conservation laws of the minimal Lagrangian system 
are defined as the elements in the 1-st  characteristic cohomology of the quotient complex
of the infinitely prolonged differential system $(\xinfh,\iinfh)$,
$$(\Omega^*(\xinfh)/\iinfh, \underline{\ed}),$$
where $\underline{\ed}=\ed\mod\iinfh.$
 
In this section, we give a  description of the infinite sequence of higher-order conservation laws
generated by the canonical formal Killing fields $\tb{X}(p^4), \tb{X}(a^5).$ 
\sub{Definition}
Let $(\Omega^*(\xinfh), \ed)$ be the de-Rham complex of $\C$-valued differential forms on $\xinfh$.
Let $$(\underline{\Omega}^*=\Omega^*(\xinfh)/\iinfh, \underline{\ed})$$ be the quotient space
equipped with the induced differential $\underline{\ed}=\ed\mod\iinfh$.
The prolongation sequence of Pfaffian systems $\Ih{k}$ satisfy the inductive closure conditions
$$\ed\Ih{k}\equiv 0\mod \Ih{k+1}, \; k\geq 1.$$
It follows that  $\iinfh=\cup_{k=0}^{\infty}\Ih{k}$ is formally Frobenius, and
$(\underline{\Omega}^*, \underline{\ed})$ 
becomes a complex.
Let $H^{q}(\underline{\Omega}^*,\, \underline{\ed})$ be the cohomology   
at $\underline{\Omega}^q.$  The set 
$$\{ \;H^{q}(\underline{\Omega}^*,\, \underline{\ed}) \;\}_{q=0}^2$$
 is called the \emph{characteristic cohomology} of the differential system $(\xinfh,\iinfh)$.
\begin{defn}
Let $(\xinfh,\iinfh)$ be the triple cover
of the infinite prolongation of the differential system for minimal Lagrangian surfaces.
A \tb{conservation law} is an element of the 1-st characteristic cohomology $H^1(\underline{\Omega}^*,\, \underline{\ed})$ of $(\xinfh,\iinfh)$. The $\C$-vector space of  conservation laws is denoted by
\[ \mcc^{(\infty)}:=H^1(\underline{\Omega}^*,\, \underline{\ed}).\]
Let $\mcc^{(\infty)}_{loc}$ denote  the space of local  conservation laws  of $\iinfh$
restricted to a small contractible open subset of $\xinfh$.
\end{defn}
For simplicity, we shall suppress the global issues 
and identify $\mcc^{(\infty)}\simeq\mcc^{(\infty)}_{loc}.$ 

Note by definition that the classical conservation laws $\mcc^{(0)}\subset\mcc^{(\infty)}.$
\subb{Spectral sequence}
Consider the filtration by the subspaces
$$ F^p\Omega^q =\tn{Image}\{\underbrace{\iinfh\w\iinfh\, ... }_{p}\w:\Omega^*(\xinfh)\to\Omega^{q}(\xinfh)\}.$$
From the associated graded
$F^p\Omega^*/F^{p+1}\Omega^*,$
a standard construction yields the spectral sequence
$$(E^{p,q}_r, \ed_r), \quad \ed_r \;\tn{has bidegree}\;(r, 1-r), \quad r\geq 0.$$

From the fundamental theorem \cite[p562, Theorem 2 and Eq.(4)]{Bryant1995}, 
at least locally the following sub-complex is exact,
\be\label{eq:exactsequence}
0\to E^{0,1}_1\hook E^{1,1}_1\to E^{2,1}_1.
\ee
Here, by definition, the first piece is given by
\begin{align*}
E^{0,1}_1 &=\{ \varphi\in\Omega^1(\xinfh) \vert \ed\varphi\equiv 0\mod\iinfh\}/
\{ \ed\Omega^0(\xinfh) +\Omega^1(\iinfh)\} \\
&=H^1(\Omega^*(\xinfh)/\iinfh, \underline{\ed})\n\\
&=\mcc^{(\infty)},\n
\end{align*}
and it is the space of conservation laws.

\subb{$E^{1,1}_1$}\label{sec:cvsymbol}
The second piece $E^{1,1}_1$ is called the space of cosymmetries. 
In the present case, the differential system is formally self-adjoint and
this is the space of Jacobi fields,  
$$E^{1,1}_1=\mfj^{(\infty)}.$$
The differential $\ed_1: E^{0,1}_1\hook E^{1,1}_1$ can be considered as the symbol map
for conservation laws.

The space $E^{1,1}_1$ admits the following analytic description.
Let $\Phi$ be a 2-form which represents a class in $E^{1,1}_1$.
By definition, one may write
$$\Phi\equiv A \Psi -\theta_0 \w \sigma \mod F^2\Omega^2,
$$
for a scalar coefficient $A$ and a 1-form $\sigma$,  where 
$$\Psi=\tn{Im}(\theta_1\w\xi)=-\frac{\im}{2}(\theta_1\w\xi-\thetab_1\w\xib).
$$
Recall
\begin{align}\label{1dPsi}
\ed\Psi&=  3\im\gamma^2\theta_0\w(\xi\w\xib+\theta_1\w\thetab_1)     \\
&\equiv 0\mod\theta_0.     \n
\end{align}

We wish to show that the coefficient $A$ is a Jacobi field.
Differentiating  $\Phi$, one gets
$$0\equiv \ed A\w\Psi-\ed\theta_0\w\sigma \mod \theta_0, F^2\Omega^3.
$$
Since $\ed\theta_0=-\frac{1}{2}(\theta_1\w\xi+\thetab_1\w\xib)$, this implies that
$$\sigma\equiv - \im\left((\delx A)\xi-(\delxb A)\xib\right)\mod\iinfh.
$$
With the given  $\sigma$, the coefficient of $\theta_0\w\xi\w\xib$-term in $\ed\Phi$
then shows that 
$$\mce(A)=0$$ and $A$ is a Jacobi field.
We thus have the isomorphism $E^{1,1}_1\simeq \mfj^{(\infty)}.$

\subsection{Conservation laws from formal Killing fields}\label{sec:cvlawfromFK}
A question arises as to if the infinite sequence of Jacobi fields 
for the minimal Lagrangian system indeed 
correspond to the sequence of conservation laws, i.e.,
if the symbol map 
$$\ed_1: \mcc^{(\infty)}_{loc}=E^{0,1}_1\hook E^{1,1}_1$$ is surjective and
a higher-order version of Noether's theorem holds for the minimal Lagrangian system.
We show that, from the two formal Killing fields constructed in the previous section,
we are able to assemble an infinite sequence of higher-order conservation laws. 
It is likely that they are nontrivial and,
considering their spectral weights,  
 the higher-order Noether's theorem holds for the minimal Lagrangian system.

\two\noi
[\tb{Formal Killing field $\tb{X}(p^4)$}]
Recall the structure equation \eqref{eq:formalKilling_n}.
Set
\be\label{eq:varphin}
\varphi_n:=b^{6n+5}\xi+s^{6n+3}\xib.
\ee
The structure equation  shows that
\[ \ed\varphi_n\equiv 0 \mod\iinfh,
\]
and $\varphi_n$ represents a conservation law.

\two\noi
[\tb{Formal Killing field $\tb{X}(a^5)$}]
Recall the structure equation \eqref{eq:formalKilling_n'}.
Set
\be\label{eq:varphin'}
\varphi'_n:=b^{6n+3}\xi+s^{6n+1}\xib.
\ee
The structure equation shows that
\[ \ed\varphi'_n\equiv 0 \mod\iinfh,
\]
and $\varphi'_n$ represents a conservation law.

\begin{thm}\label{thm:FKcvlaw}
Let $\tb{X}(p^4), \tb{X}(a^5)$ be the formal Killing fields generated from the initial data
$p^4=z_4, a^5=z_5-\frac{5}{3}z_4^2$ respectively in \S\ref{sec:formalKilling}.
Then the associated sequence of 1-forms 
$$\varphi_n, \; \varphi'_n, \quad n=0, 1, 2, \, ... \, ,$$
represent the higher-order conservation laws.
\end{thm}
It remains to verify that these conservation laws are indeed nontrivial.
But, Thm.\ref{thm:FKcvlaw}  points to the relevant questions such as periods, residues, 
and the related application of the higher-order conservation laws 
to the global problems for minimal Lagrangian surfaces.
\begin{rem}
Suppose the conservation laws $[\varphi_n], [\varphi_n']$ are nontrivial.
Then, by an analysis of the differential $\ed _1$ as in \S\ref{sec:cvsymbol},
the   spectral weight count shows  that, 
\[
\ed_1([ \varphi_n]) =a^{6n+7},\quad
\ed_1([ \varphi_n'])  =a^{6n+5}, \]
up to constant scale, where $\ed_1$ is the symbol map
$\ed_1:\mcc^{(\infty)}_{loc}\to E^{1,1}_1\simeq \mfj^{(\infty)}$.
\end{rem}

\providecommand{\MR}[1]{}
\providecommand{\bysame}{\leavevmode\hbox to3em{\hrulefill}\thinspace}
\providecommand{\MR}{\relax\ifhmode\unskip\space\fi MR }
\providecommand{\MRhref}[2]{%
  \href{http://www.ams.org/mathscinet-getitem?mr=#1}{#2}
}
\providecommand{\href}[2]{#2}

\end{document}